\documentclass[a4paper]{article}
\usepackage[margin=1in]{geometry}
\usepackage{amssymb}
 \usepackage{wrapfig}
 \usepackage{flafter}
 \usepackage{leftidx}
 \usepackage{mathrsfs}
 \usepackage{amssymb,amsmath}
 \usepackage{amsthm}
 \usepackage{bbding}
\usepackage{fancyhdr}%è°ç¨æå°é¡µçä¸é¡µèçå®å
\usepackage{mathrsfs}
\usepackage{stmaryrd}
\usepackage{amsfonts}
\usepackage{latexsym}
\usepackage{psfrag}
\usepackage{graphicx,subfigure}
\usepackage{dsfont}
\usepackage{bbm}
\usepackage{booktabs}
\usepackage{color,xcolor}
\usepackage{authblk}
\usepackage{appendix}
\usepackage{cite}
\usepackage{epstopdf}
\allowdisplaybreaks

\usepackage{graphicx}
\usepackage{natbib,color}
\usepackage{enumitem}

\newtheorem{theorem}{Theorem}[section]
\newtheorem{lemma}[theorem]{Lemma}

\newtheorem{remark}[theorem]{Remark}
\newtheorem{proposition}[theorem]{Proposition}

\graphicspath{{./}}

\begin{document}

\title{Optimal error bounds on an exponential wave integrator Fourier spectral method for fractional nonlinear Schr\"{o}dinger equations with low regularity potential and nonlinearity}
% Short title for running heads:
%\shorttitle{Optimal error bounds on EWI-FS for SFNLSE}

\author{%
{\sc Junqing Jia}\thanks{ Email: 202390000016@sdu.edu.cn; matv178@nus.edu.sg}
 \\
School of Mathematics, Shandong University, Jinan 250100, China\\
Department of Mathematics, National University of Singapore, Singapore 119076, Singapore\\
{\sc Xiaoyun Jiang}\thanks{Corresponding author. Email: wqjxyf@sdu.edu.cn}
 \\
School of Mathematics, Shandong University, Jinan 250100, China\\
}
% Short list of authors for running heads:
%\shortauthorlist{J. Jia \emph{et al.}}

\maketitle

\begin{abstract}
% Body of abstract:
{We establish optimal error bounds on an exponential wave integrator (EWI) for the space fractional nonlinear Schr\"{o}dinger equation (SFNLSE) with low regularity potential and/or nonlinearity. For the semi-discretization in time, under the assumption of $L^\infty$-potential, $C^1$-nonlinearity, and $H^\alpha$-solution with $1<\alpha \leq 2$ being the fractional index of $(-\Delta)^\frac{\alpha}{2}$, we prove an optimal first-order $L^2$-norm error bound $O(\tau)$ and a uniform $H^\alpha$-norm bound of the semi-discrete numerical solution, where $\tau$ is the time step size. We further discretize the EWI in space by the Fourier spectral method and obtain an optimal error bound in $L^{2}$-norm $O(\tau+h^{m})$ without introducing any CFL-type time step size restrictions, where $h$ is the spatial step size, $m$ is the regularity of the exact solution. Moreover, under slightly stronger regularity assumptions, we obtain optimal error bounds $O(\tau)$ and $O(\tau+h^{m-{\frac{\alpha}{2}}})$ in $H^\frac{\alpha}{2}$-norm, which is the norm associated to the energy. Extensive numerical examples are provided to validate the optimal error bounds and show their sharpness. We also find distinct evolving patterns between the SFNLSE and the classical nonlinear Schr\"{o}dinger equation.}

\end{abstract}

\section{Introduction}
In this paper, we consider the following space fractional nonlinear Schr\"{o}dinger equation (SFNLSE) \citep{Zhai2019} on a bounded domain $\Omega=\Pi_{j=1}^d \left(a_j, b_j\right)^{d} \subset \mathbb{R}^d \ (d=1,2,3)$ equipped with periodic boundary condition:
\begin{equation}\label{e1.1}
\left\{
\begin{aligned}
	&i \partial_t \psi(\boldsymbol{x}, t)=(-\Delta)^{\frac{\alpha}{2}} \psi(\boldsymbol{x}, t)+V(\boldsymbol{x}) \psi(\boldsymbol{x}, t)+f\left(|\psi(\boldsymbol{x}, t)|^2\right) \psi(\boldsymbol{x}, t),&& \boldsymbol{x} \in \Omega, \\
	&\psi(\boldsymbol{x}, 0)=\psi_0(\boldsymbol{x}), && \boldsymbol{x} \in \overline{\Omega},
\end{aligned}
\right.
\end{equation}
where ${i}$ is the imaginary unit, $\boldsymbol{x} = (x_1, \cdots, x_d)^{\rm T} \in \Omega$ with $\boldsymbol{x} = x$ when $d=1$ is the spatial coordinate, $t \geq 0$ is time, and $\psi=\psi(\boldsymbol{x}, t)$ is a complex-valued wave function. The operator $(-\Delta)^{\frac{\alpha}{2}}(1<\alpha\leq2)$ is the space fractional Laplacian \citep{Ainsworth2017} defined in the phase space through its action on the Fourier coefficients:
\begin{equation*}\label{eq1.2}
\widehat{\left((-\Delta)^{\frac{\alpha}{2}} \phi\right)}_{\boldsymbol{k}}=|\mu_{\boldsymbol{k}}|^\alpha \widehat{\phi}_{\boldsymbol{k}}, \quad (\mu_{\boldsymbol{k}})_j=\frac{2 \pi k_j}{b_j - a_j}, \quad \boldsymbol{k} = (k_1, \cdots, k_d)^{\rm T} \in \mathbb{Z}^d,
\end{equation*}
where $| \cdot |$ is the Euclidean norm on $\mathbb{R} ^d$, and $\widehat{\phi}$ is the Fourier transform defined as
\begin{equation}\label{eq:FT}
	\widehat{\phi}_{\boldsymbol{k}}=\frac{1}{|\Omega|} \int_{\Omega} \phi(\boldsymbol{x}) {e}^{-i \mu_{\boldsymbol{k}} \cdot (\boldsymbol{x}-\boldsymbol{a})} \rm d \boldsymbol{x}, \quad \boldsymbol{k} \in \mathbb{Z}^d,
\end{equation}
with $|\Omega| = \Pi_{j=1}^d(b_j-a_j)$ being the volume of $\Omega$ and $\boldsymbol{a} = (a_1, \cdots, a_j)^{\rm T}$. The potential $V=V(\boldsymbol{x})$ is real-valued and time-independent, and the nonlinearity $f = f(\rho)$ with $\rho = |\psi|^2$ being the density is also real-valued. We shall focus on low regularity potential and/or nonlinearity in this work, and we assume that $V \in L^\infty(\Omega)$ and $f$ is the typical power-type nonlinearity of the form
\begin{equation}\label{eq:f}
	f(\rho) = \beta \rho^\sigma, \qquad \beta \in \mathbb{R}, \quad \sigma>0,
\end{equation}
where $\beta$ is a given parameter characterizing the nonlinear interaction ($\beta > 0$ for repulsive interaction and $\beta < 0$ for attractive interaction), and $\sigma$ is the exponent of nonlinearity. Note that when $0<\sigma<\frac{1}{2}$, the nonlinearity $f(|\psi|^2)\psi$ is only $C^1$ in the real sense.

When $\alpha=2, f(\rho)=\beta \rho$, the SFNLSE \eqref{e1.1} reduces to the classcial cubic nonlinear Schr\"{o}dinger equation (NLSE) or the Gross-Pitaevskii equation (GPE), which has been extensively employed for modeling and simulation in Bose-Einstein condensation (BEC) \citep{Bao2013a}, quantum mechanics \citep{Erdos2007}, and nonlinear optics \citep{Sulem1999}. And it has been thoroughly investigated using various numerical methods, including the finite difference method \citep{Akrivis1993,Antoine2013,Bao2013b}, the exponential wave integrator (EWI) \citep{Bao2014a,Celledoni2008}, the time-splitting method \citep{Antoine2013,Bao2023a,Bao2022,Besse2002a,Lubich2008a}, the finite element method \citep{Akrivis1991,Henning2017,Tourigny1991}, etc.

In recent decades, the SFNLSE \eqref{e1.1} has found numerous applications in various physical phenomena, including water wave dynamics \citep{Ionescu2014}, L\'{e}vy processes \citep{Laskin2000}, optics \citep{Longhi2015}, and quantum mechanics \citep{Pinsker2015} and so on. The nonlocality of the SFNLSE allows for the description of phenomena distinct from those modeled by the classical NLSE \citep{Duo2016,Longhi2015}. When the nonlinearity is of the Hartree type, it has been employed to model the dynamics of Boson stars \citep{Bao2011,Feng2021c}. For further details on the physical background and theoretical analysis of the SFNLSE, we refer the reader to \citep{Antoine2016,Benson2011,Klein2014,Zhang2015} and the references therein.

There is an extensive body of both analytical and numerical results in the literature concerning the SFNLSE \eqref{e1.1}. From an analytical standpoint, exact solutions and the local and/or global well-posedness of the Cauchy problem have been comprehensively studied in \citep{Guo2008,Jiang2011}. On the numerical side, a variety of accurate and efficient numerical methods have been proposed and analyzed \citep{Feng2021c,Fu2023,Jia2023,Li2018,Wang2015,Zhai2019}. Additionally, for numerical studies of other fractional equations, we can refer to \citep{Acosta2018,Ervin2018,Mustapha2018,Zhang2023} and the references therein.
In previous studies of the SFNLSE, both $V$ and $f$ in \eqref{e1.1} are generally assumed to be smooth. However, in various physical applications, $V$ and $f$ may exhibit low regularity. For instance, in several physical contexts, the low-regularity $L^{\infty}$-potentials often take the form of discontinuous square-well or step potentials. In the study of BEC under different trapping shapes, a commonly used potential is the power-law potential $V(\boldsymbol{x})=$ $|\boldsymbol{x}|^\gamma(\gamma>0)$ \citep{BaYindir1999,Pinkse1997}. Furthermore, disorder potentials play a critical role in analyzing phenomena such as the Josephson effect and Anderson localization \citep{Sanchez-Palencia2007,Zapata1998}. Low-regularity nonlinearities, such as $f(\rho)=\rho^\sigma(\sigma>0)$ or $f(\rho)=\rho \ln \rho$, also arise in various models. These include the Schr\"{o}dinger-Poisson-X$\alpha$ model \citep{Bao2003,Bokanowski1999}, the Lee-Huang-Yang correction \citep{Lee1957}, which has been used to model and simulate quantum droplets \citep{Astrakharchik2018,Petrov2016}, and the mean-field model for Bose-Fermi mixtures \citep{Cai2013,Hadzibabic2002}.

For the classical NLSE with low regularity potential and/or nonlinearity, and some low regularity problems with rough solutions, several studies have been conducted \citep{Bao2023,Bao2024,Bao2024sEWI,Henning2017,Li2023,Li2024}. In particular, Bao and Wang \citep{Bao2024} establish optimal error bounds on a first-order EWI for the NLSE with $L^{\infty}$-potential and/or locally Lipschitz nonlinearity. These error bounds significantly improve the existing results for the NLSE with low regularity potential and/or nonlinearity. In contrast to the traditional $H^{1}$-norm framework, our work addresses a significant theoretical challenge by developing a novel analysis of the error in the $H^{\frac{\alpha}{2}}$-norm. Additionally, we rigorously establish a previously unexplored inequality between the nonlinearity exponent $\sigma$ and the fractional power $\alpha$, thus advancing the theory of fractional Schr\"{o}dinger equations with low regularity potentials and/or nonlinearities. Moreover, using the Fourier spectral method to further discretize the EWI in space, we show that the spatial error is always optimal with respect to the regularity of the exact solution.

For simplicity of the presentation, we carry out the analysis in the one-dimensional (1D) setting. The results can be easily generalized to two dimensions (2D) and three dimensions (3D) with modifications, for more details, please refer to Remark \ref{rem4.4} and the Appendix at the end of the article. For the convenience of the reader, we summarize our main results in 1D as follows:
\begin{enumerate}[label=(\roman*)]
\item For the semi-discretization in time, assuming $L^\infty$-potential, $\sigma > 0$, and $H^\alpha$-solution, we establish an optimal $O(\tau)$ convergence in $L^2$-norm together with a uniform $H^{\alpha}$-norm bound of the numerical solution, which implies a  $O(\tau^{\frac{1}{2}})$ convergence in the energy norm $H^\frac{\alpha}{2}$ (Theorem \ref{th3.1}).

\item For the fully discrete scheme with the Fourier spectral method for spatial discretization, under the same regularity assumptions as above, we show, in addition to the temporal convergence stated above, the spatial convergence is of $O(h^{m})$ order in $L^2$-norm and of $O(h^{m-\frac{\alpha}{2}})$ order in $H^\frac{\alpha}{2}$-norm, where $m \geq \alpha$ is the actual regularity of the exact solution (Theorem \ref{th4.1}). Notably, this result is obtained without imposing any coupling conditions between the time step size $\tau$ and the mesh size $h$.

\item Assuming stronger regularity of $H^\frac{\alpha}{2}$-potential, $\sigma > \frac{\alpha}{4}$, and $H^{\frac{3\alpha}{2}}$-solution or $H^{m}$-solution $(m\geq\frac{3\alpha}{2})$, we also prove optimal convergence orders $O(\tau)$ and $O(\tau+h^{m-{\frac{\alpha}{2}}})$ in the energy norm $H^\frac{\alpha}{2}$ for both the semidiscretiztion and the full discretization (Theorem \ref{th3.1} and \ref{th4.1}), respectively.

\end{enumerate}

The remainder of this paper is organized as follows. Section 2 shows the numerical method, first, we give some preliminaries, then we introduces the semi-discretization in time based on the first-order EWI, followed by a full discretization using the Fourier spectral method in space. In Section 3, we present some relevant lemmas necessary for the theoretical framework. Section 4 rigorously establishes the error estimates for both the semi-discretization and full discretization schemes. Section 5 provides numerical results that validate the theoretical error estimates and highlight significant differences in the evolving patterns between the SFNLSE and the classical NLSE. Finally, Section 6 concludes the study. Throughout the paper, we use $C(M)$ to denote a generic positive constant depending on the parameter $M$ but independent of the time step size $\tau$ and the mesh size $h$. The notation $A \lesssim B$  implies the existence of a generic constant $C>0$ such that $|A| \leq C B$.

\section{Numerical method}
As mentioned before, for simplicity,  we present the numerical scheme and the error analysis in 1D case, and take $\Omega=(a, b)$. The generalizations to 2D and 3D are straightforward (see the discussion in Remark \ref{rem4.4} ).
\subsection{Preliminary}
Let $\tau>0$ be the time step and $h=(b-a)/N$ be the mesh size, where $N$ is an even positive integer. The grid points are denoted as
\begin{center}
$t_{n}:=n\tau, \quad n=0,1,2,\cdots; \quad x_{j}:=a+jh, \quad j=0,1,2,\cdots,N$.
\end{center}
Define the index set
\begin{equation*}
{T}_N:=\left\{-\frac{N}{2},-\frac{N}{2}+1, \ldots, \frac{N}{2}-1\right\},
\end{equation*}
and denote
\begin{equation*}
C_{\textrm{per}}(\Omega):=\{\phi\in C(\overline{\Omega}) : \phi(a)=\phi(b)\},
\end{equation*}
\begin{equation*}
X_{N}:=\{\phi=(\phi_{0},\phi_{1},\ldots, \phi_{N})^{T}\in\mathbb{C}^{N+1} : \phi_{0}=\phi_{N}\},
\end{equation*}
\begin{equation*}
Y_{N}:=\textrm{span}\left\{{e}^{i\mu_{l}(x-a)}: l \in {T}_N\right\}, \quad \mu_{l}=\frac{2\pi l}{b-a}.
\end{equation*}
Let $P_{N}: L^{2}(\Omega)\rightarrow Y_{N}$ be the standard $L^{2}$-projection operator, $I_{N}:C_{\textrm{per}}(\Omega)\rightarrow Y_{N}$ or $I_{N}:X_{N}\rightarrow Y_{N}$ be the trigonometric interpolation operator \citep{Canuto2006}, i.e.,
\begin{equation}\label{eq2.1.1}
(P_{N}\phi)(x)=\sum_{l \in  {T}_N}\widehat{\phi}_{l}{e}^{i\mu_{l}(x-a)}, \quad (I_{N}\phi)(x)=\sum_{l \in {T}_N}\widetilde{\phi}_{l}{e}^{i\mu_{l}(x-a)}, \quad x\in \overline{\Omega},
\end{equation}
where $\widehat{\phi}$ is the Fourier transform of $\phi$ defined in \eqref{eq:FT} and $\widetilde{\phi}$ is the discrete Fourier transform defined as
\begin{equation}\label{eq2.1.2}
\widetilde{\phi}_{l}=\frac{1}{N}\sum_{j=0}^{N-1}\phi_j {e}^{-i\mu_{l}(x_{j}-a)},
\end{equation}
where $\phi_{j}=\phi(x_j)$ for $0 \leq j \leq N$.

For $s\geq 0$,  we denote by $H^{s}_{\mathrm{per}}(\Omega)$ the periodic Sobolev space with norm and semi-norm \citep{Canuto2006,Feng2021c}
\begin{equation}\label{eq2.1.3}
\|\phi\|_{H^s}^{2}=\sum_{l\in\mathbb{Z}}(1+\mu_{l}^{2})^{s}|\widehat{\phi_{l}}|^{2}, \quad |\phi|_{H^s}^{2}=\sum_{l\in\mathbb{Z}}|\mu_{l}|^{2s}|\widehat{\phi_{l}}|^{2}, \quad \phi \in L^2(\Omega),
\end{equation}
and $H^{0}_{\mathrm{per}}(\Omega)=L^{2}(\Omega)$.

\subsection{Temporal semi-discretization by the EWI}
To simplify the notation, we let $\psi(t) = \psi(x, t)$ and define an operator
\begin{equation}
	G(\phi) = f(|\phi|^2)\phi = |\phi|^{2\sigma}\phi, \quad \phi \in L^2(\Omega).
\end{equation}
By Duhamel's formula, the exact solution to the SFNLSE \eqref{e1.1} is expressed as
\begin{align}
\psi(t_{n+1})
&=e^{i \tau \langle\nabla\rangle_{\alpha}} \psi(t_n) \notag \\
&\quad -i \int_0^\tau e^{i(\tau-s) \langle\nabla\rangle_{\alpha}}[V \psi(t_n+s)+G(\psi(t_n+s))] \textrm{d}s, \quad n \geq 0,  \label{e2.1}
\end{align}
where $\langle\nabla\rangle_{\alpha}=-(-\Delta)^{\frac{\alpha}{2}}$.

Let $\psi^{[n]} $ be the approximation of $\psi (t_n)$ for $n \geq 0$. Adopting the approximation $\psi(t_n+s) \approx \psi(t_n)$ in the integrand of \eqref{e2.1} and performing an exact integration of ${e}^{i(\tau-s) \langle\nabla\rangle_{\alpha}}$, we obtain a semi-discretization in time by the first-order EWI as
\begin{equation}\label{e2.2}
\begin{aligned}
\psi^{[n+1]} & ={e}^{i \tau \langle\nabla\rangle_{\alpha}} \psi^{[n]}-i \tau \varphi_1(i \tau \langle\nabla\rangle_{\alpha})\left(V \psi^{[n]}+G(\psi^{[n]}) \right), \quad  n \geq 0, \\
\psi^{[0]} & =\psi_0,
\end{aligned}
\end{equation}
where, $\varphi_1$ is an entire function defined as
$$
\varphi_1(z)=\frac{{e}^z-1}{z}, \quad z \in \mathbb{C} .
$$
The operator $\varphi_1(i \tau \langle\nabla\rangle_{\alpha})$ is defined by its action in the Fourier space as
\begin{align}\label{e2.3}
(\varphi_1(i \tau \langle\nabla\rangle_{\alpha}) v)(x)
& =\sum_{l \in \mathbb{Z}} \varphi_1(-i \tau |\mu_l|^{\alpha}) \widehat{v}_l {e}^{i \mu_l(x-a)} \notag \\
& =\widehat{v}_0+\sum_{l \in \mathbb{Z} \backslash\{0\}} \frac{1-{e}^{-i \tau |\mu_l|^{\alpha}}}{i \tau |\mu_l|^{\alpha}} \widehat{v}_l {e}^{i \mu_l(x-a)}, \quad x \in \Omega.
\end{align}
We define the semi-discrete numerical flow $\Phi^\tau$ associated with the EWI \eqref{e2.2} as
\begin{equation}\label{e2.3-1}
	\Phi^\tau(\phi) = e^{i \tau \langle\nabla\rangle_{\alpha}} \phi-i \tau \varphi_1(i \tau \langle\nabla\rangle_{\alpha})\left(V \phi +G(\phi) \right), \quad \phi \in L^2(\Omega).
\end{equation}
From \eqref{e2.3}, observing that $\left|1-{e}^{-i \theta}\right| \leq 2$ for $\theta \in \mathbb{R}$, we derive the following estimate
\begin{equation}\label{e2.5}
\left|\widehat{(\varphi_1 (i \tau \langle\nabla\rangle_{\alpha}) v)}_l\right| \leq \begin{cases}
	\frac{2}{\tau} \frac{|\widehat{v}_l|}{|\mu_l|^{\alpha}}, & l \in \mathbb{Z} \backslash\{0\}, \\ |\widehat{v}_0|, & l=0,
\end{cases}
\end{equation}
which implies $\varphi_1(i \tau \langle\nabla\rangle_{\alpha}) v \in H_{\text{per}}^{\alpha}(\Omega)$ for all $v \in L^2(\Omega)$. Consequently, $\Phi^\tau$ defines a flow on $H_{\text{per}}^{\alpha}(\Omega)$ for any $V \in L^{\infty}(\Omega)$ and $\sigma >0$, enabling the establishment of a uniform $H^\alpha$-bound for the semi-discrete solution using the new analytical techniques that will be introduced later.

\subsection{Full discretization by the Fourier spectral method in space}
We further discretize the semi-discretization \eqref{e2.2} in space by the Fourier spectral (FS) method to obtain a full discretization scheme. Usually, the Fourier pseudospectral (FP) method is used for spatial discretization (see, e.g., \citep{Bao2014a,Feng2023}), which can be efficiently implemented with FFT. However, due to the low regularity of potential and nonlinearity, it is very hard to establish error estimates of the FP method, and severe order reduction can be observed numerically \citep{Bao2024a}.

Let $\psi^n \in Y_N$ be the approximation of $\psi(t_n)$ for $n \geq 0$. Then the EWI-Fourier spectral (EWI-FS) method reads
\begin{equation}\label{e2.12}
	\begin{aligned}
		\widehat{(\psi^{n+1})}_l & ={e}^{-i \tau |\mu_l|^{\alpha}}{\widehat{(\psi^n)}_l}-i \tau \varphi_1(-i \tau |\mu_l|^{\alpha})(\widehat{(V \psi^n)}_l+\widehat{G(\psi^n)}_l), \quad n \geq 0, \\
		\widehat{(\psi^0)}_l & ={\widehat{(\psi_0)}}_l, \quad l \in {T}_N.
	\end{aligned}
\end{equation}
Similarly, we define a fully discrete numerical flow $\Phi_h^\tau:Y_N \rightarrow Y_N$ associated with the EWI-FS method as
\begin{equation}
	\Phi_h^\tau(\phi) = {e}^{i \tau \langle\nabla\rangle_{\alpha}} \phi-i \tau \varphi_1(i \tau \langle\nabla\rangle_{\alpha}) P_N(V \phi+ G(\phi)), \quad \phi \in Y_N.
\end{equation}
Then the numerical solution $\psi^n (n \geq 0)\in Y_N$ obtained from \eqref{e2.12} satisfies
\begin{equation}\label{e2.11}
	\begin{aligned}
		\psi^{n+1} & ={e}^{i \tau \langle\nabla\rangle_{\alpha}} \psi^n-i \tau \varphi_1(i \tau \langle\nabla\rangle_{\alpha}) P_N(V \psi^n+ G(\psi^n)), \quad n \geq 0, \\
		\psi^0 & =P_N \psi_0.
	\end{aligned}
\end{equation}

\begin{remark}[Implementation]
To implement the EWI-FS method \eqref{e2.12} in practice, one needs to accurately compute the Fourier coefficients of $V\psi^n$ and $G(\psi^n)$. This can be done efficiently by using an extended FFT as in \citep{Bao2024a}. For the convenience of the reader, we illustrate this process below.

For any integer $K \geq N$, define $K$ equally distributed quadrature points as
\begin{equation*}
x_j^K: =a+j \frac{b-a}{K}, \quad j=0,1,2, \cdots, K.
\end{equation*}
For $\phi \in L^{\infty}(\Omega)$ bounded on $\overline{\Omega}$, the $l$-th Fourier coefficient $\widehat{\phi}_l$ can be approximated by
\begin{equation*}
\widehat{\phi}_l \approx \widetilde{\phi}_{K, l}:=\frac{1}{K} \sum_{j=0}^{K-1} \phi(x_j^K) e^{-i \mu_l(x_j^K-a)}, \quad l \in {T}_N .
\end{equation*}
Note that $\widetilde{\phi}_{K, l}=\widetilde{\phi}_l$ for $l \in T_N$, where
\begin{equation*}
\phi=(\phi_0, \cdots, \phi_K)^T \in X_K, \quad \phi_j=\phi(x_j^K), \quad j=0,1,2, \cdots, K.
\end{equation*}
Then, we use $\widehat{G(\psi^n)}_l \approx \widetilde{G(\psi^n)}_{K,l}$ for $l \in {T}_N$ with $K \geq N$ in \eqref{e2.12}. This method is called the Fourier spectral method with quadrature (FSwQ) \citep{Bao2024a}. According to our extensive numerical results, it suffices to choose $K=3N$ to obtain the desired optimal convergence rate.

For the computation of $\widehat{(V \psi^n)}_l$, we use an extended Fourier pseudospectral (eFP) method \citep{Bao2024a,wangthesis}. When $\psi^n \in Y_N$, we have, see \citep{Bao2024a},
\begin{equation}\label{e2.14}
P_N(V \psi^{ n}) = P_N((P_{2 N}V) \psi^{ n }) =P_N(I_{3 N}((P_{2 N}V) \psi^{n})),
\end{equation}
which implies
\begin{equation}\label{e2.14-1}
\widehat{(V\psi^{ n})}_l=\widetilde{U}_{l}, \quad l\in T_{N},
\end{equation}
where $U\in X_{3N}$ is defined as
\begin{equation}\label{e2.14-2}
\begin{aligned}
U_{j}=(P_{2N}V)(x_{j}^{3N})\times\psi(x_{j}^{3N}),\quad j=0,1,\cdots,3N.
\end{aligned}
\end{equation}
Here, $P_{2N}V$ can be precomputed numerically or analytically.

\end{remark}

\section{Some useful lemmas}
Here, we present some useful lemmas that will be used in the subsequent proof of the optimal error estimates.

We start with an estimate for $f$ defined in \eqref{eq:f}.
\begin{lemma}\label{assum1}
Let $z_1, z_2 \in \mathbb{C}$ such that $|z_1| \leq M_0$ and $|z_2| \leq M_0$, we have
\begin{equation*}
|f(|z_1|^2) z_1-f(|z_2|^2) z_2| \leq C(M_0)|z_1-z_2|.
\end{equation*}
\end{lemma}
\begin{proof}
For any $z_1, z_2 \in \mathbb{C}$, let $z^\theta=(1-\theta) z_1+\theta z_2$ and let $\gamma(\theta)=f(|z^\theta|^2) z^\theta$ for $0 \leq \theta \leq 1$, we have
\begin{equation}\label{eq3.6}
	f(|z_2|^2) z_2-f(|z_1|^2) z_1=\gamma(1)-\gamma(0)=\int_0^1 \gamma^{\prime}(\theta) \mathrm{d} \theta,
\end{equation}
where
\begin{equation}\label{eq:gamma}
	\gamma^{\prime}(\theta)=(1+\sigma) f(|z^\theta|^2)(z_2-z_1)+g(z^\theta) \overline{(z_2-z_1)},
\end{equation}
with
\begin{equation}\label{eq3.7-n}
	g(z)=\left\{\begin{aligned}
		&f^{\prime}\left(|z|^2\right) z^2=\beta \sigma|z|^{2 \sigma-2} z^2, && z \neq 0, \\
		&0, && z=0,
	\end{aligned} \quad z \in \mathbb{C}.\right.
\end{equation}
Then we have, since $\sigma > 0$,
\begin{equation}
	|f(|z_2|^2) z_2-f(|z_1|^2) z_1| \leq \int_0^1 |\gamma^{\prime}(\theta)| \mathrm{d} \theta \lesssim |z_\theta|^{2\sigma}|z_1 - z_2| \leq M_0^{2\sigma} |z_1 - z_2|,
\end{equation}
which completes the proof.
\end{proof}

We say $F:\mathbb{C} \rightarrow \mathbb{C}$ is (globally) $p$-H\"older continuous with $0<p\leq1$ if
\begin{equation*}
	| F(z_1) - F(z_2) | \leq C |z_1 - z_2|^{p}, \quad z_1, z_2 \in \mathbb{C},
\end{equation*}
where $C$ is a constant independent of $z_1$ and $z_2$.

\begin{lemma}\label{lem:alphanorm}
	Let $\frac{\alpha}{4} < \sigma \leq \frac{1}{2}$ and $F$ be $2\sigma$-H\"older continuous. For any $v \in H_\text{\rm per}^1(\Omega)$, we have $F(v) \in H_\text{\rm per}^\frac{\alpha}{2}(\Omega)$ and
	\begin{equation}
		\| F(v) \|_{H^\frac{\alpha}{2}} \lesssim \| v \|_{H^1}^{2\sigma}.
	\end{equation}
\end{lemma}

\begin{proof}
	From Lemma 2.1 in \citep{Zhang2024}, recalling that $\Omega$ is equipped with the periodic boundary condition, we have
	\begin{equation}\label{eq3.9}
		\begin{aligned}
			|F(v)|_{H^{\frac{\alpha}{2}}}^2&\lesssim\int_{\Omega}\int_{\Omega}\frac{|F(v(x+y))-F(v(y))|^2}{|x|^{1+\alpha}}\mathrm{d} x \mathrm{d} y\\
			&\lesssim\int_{\Omega}\int_{\Omega}\frac{|v(x+y)-v(y)|^{4\sigma}}{|x|^{1+\alpha}}\mathrm{d} x \mathrm{d} y \quad
			\text{(\emph{F} is $2\sigma-$H\"{o}lder continuous)}\\
			&=\int_{\Omega}\int_{\Omega}\frac{|\int_{0}^{1}v'(y+\theta x)x\mathrm{d} \theta|^{4\sigma}}{|x|^{1+\alpha}}\mathrm{d} x \mathrm{d} y \\
			&\leq \int_{\Omega}\int_{\Omega}\int_{0}^{1}|v'(y+\theta x)|^{4\sigma}\mathrm{d} \theta\frac{|x|^{4\sigma}}{|x|^{1+\alpha}}\mathrm{d} x \mathrm{d} y\\
			&\lesssim\int_{\Omega}\frac{|x|^{4\sigma}}{|x|^{1+\alpha}}\mathrm{d} x\int_{\Omega} \int_{0}^{1} |v'(y)|^{4\sigma}\mathrm{d} \theta \mathrm{d} y \quad \text{(change integration order)}\\
			&=\int_{\Omega}\frac{|x|^{4\sigma}}{|x|^{1+\alpha}}\mathrm{d} x \ \|v'\|^{4\sigma}_{L^{4\sigma}}\\
			&\lesssim \int_{\Omega}\frac{|x|^{4\sigma}}{|x|^{1+\alpha}}\mathrm{d} x \ \|v\|^{4\sigma}_{H^1} \quad \text{($\sigma\leq \frac{1}{2}$)}\\
			&\lesssim \|v\|^{4\sigma}_{H^1} \quad \text{($\sigma>\frac{\alpha}{4}$, the integral $\int_{\Omega}\frac{|x|^{4\sigma}}{|x|^{1+\alpha}}\mathrm{d} x$ converges)}.
		\end{aligned}
	\end{equation}
	Thus, we have completed the proof.
\end{proof}

\begin{lemma}\label{assum2}
When $\sigma > \frac{\alpha}{4}$, for $v, w \in H_\text{\rm per}^{1}(\Omega)$,
\begin{equation*}
\begin{aligned}
\|f(|v|^2) v-f(|w|^2) w\|_{H^\frac{\alpha}{2}}
\leq C(\|v\|^{2\sigma}_{H^1},\|w\|^{2\sigma}_{H^1})\| v-w \|_{H^\frac{\alpha}{2}}.
\end{aligned}
\end{equation*}
\end{lemma}

\begin{proof}

Similar to \eqref{eq3.6} with $z^\theta = (1-\theta)v + \theta w$ for $0 \leq \theta \leq 1$, we have
\begin{equation}\label{eq3.8}
\|f(|v|^2) v-f(|w|^2) w \|_{H^\frac{\alpha}{2}} \leq \int_0^1 \| \gamma^{\prime}(\theta)\|_{H^\frac{\alpha}{2}} \rm d \theta.
\end{equation}
Since $\frac{\alpha}{2}>\frac{1}{2}$, by the algebra property of $H^\frac{\alpha}{2}$ in 1D, and recalling \eqref{eq:gamma}, we have
\begin{equation}\label{eq:gamma-derivative}
	\| \gamma^{\prime}(\theta)\|_{H^\frac{\alpha}{2}} \lesssim \left(\| f_1(z^\theta) \|_{H^\frac{\alpha}{2}} + \| f_2(z^\theta) \|_{H^\frac{\alpha}{2}}\right) \| v - w \|_{H^\frac{\alpha}{2}},
\end{equation}
where $f_{1}(z)=|z|^{2\sigma}$ and $f_2(z)=|z|^{2\sigma-2}z^2$ for $z \in \mathbb{C}$ and we define $(f_j(z^\theta))(x) = f_j(z^\theta(x))$ for $j=1,2$ and $x \in \Omega$. To apply Lemma \ref{lem:alphanorm}, we need to show that $f_1$ and $f_2$ are $2\sigma$-H\"older continuous, which is obvious for $f_{1}(z)$ with $0<\sigma \leq 1/2$. In the following, we show that $f_{2}(z)$ is also $2\sigma-$H\"{o}lder continuous, i.e.,
\begin{equation}\label{eq:f2_Holder}
	\left||z_{1}|^{2\sigma-2}z_{1}^2-|z_{2}|^{2\sigma-2}z_{2}^2\right|\leq C|z_{1}-z_{2}|^{2\sigma}, \quad 0 < \sigma \leq \frac{1}{2}.
\end{equation}
Without loss of generality, we assume that $0<|z_2|\leq|z_1|$. Since \eqref{eq:f2_Holder} can be rewritten as
\begin{equation}\label{eq3.38-2}
\left||z_{1}|^{2\sigma-2}z_{1}^2\left(1-\frac{|z_{2}|^{2\sigma-2}z_{2}^2}{|z_{1}|^{2\sigma-2}z_{1}^2}\right)\right|
\leq C|z_{1}-z_{2}|^{2\sigma},
\end{equation}
it suffices to prove that
\begin{equation}\label{eq:f2_Holder_reduce}
	\left|1-\left|z\right|^{2\sigma-2}z^2\right|\leq C\left|1-z\right|^{2\sigma},
\end{equation}
where $z=z_2/z_1 \in \mathbb{C}$ satisfying $0<|z|\leq 1$. Let $z=re^{i\theta}$ with $ 0<r \leq 1 $ and $ -\pi \leq \theta < \pi$. Then \eqref{eq:f2_Holder_reduce} becomes
\begin{equation}\label{eq3.38-5}
|1-r^{2\sigma}{e}^{2i\theta}|\leq C|1-r{e}^{i\theta}|^{2\sigma}.
\end{equation}
We only need to consider cases where $z$ is close to $1$, i.e., $(r, \theta) $ is close to $ (1, 0)$. For the LHS, using $1-r^{2\sigma} \leq (1-r)^{2\sigma}$ for $0<r\leq1$ and $0<\sigma \leq 1/2$, we have
\begin{equation}\label{eq:LHS}
	|1-r^{2\sigma}e^{2i\theta}| \leq |1-r^{2\sigma}| + r^{2\sigma} |1-e^{2i\theta}| \leq (1-r)^{2\sigma} +2 |\theta| \leq (1-r)^{2\sigma} +2 |\theta|^{2\sigma}.
\end{equation}
For the RHS, since $(x+y)^\sigma \geq (x^\sigma+y^\sigma)/2$ for $x\geq0, y\geq0$, we have
\begin{align}
	|1-r e^{i\theta}|^{2\sigma}
	&= (1+r^2 - 2r\cos \theta)^{\sigma} = [(1-r)^2 + 2r(1-\cos\theta)]^\sigma \notag \\
	&\geq \frac{1}{2}(1-r)^{2\sigma} + \frac{1}{2} 2^\sigma r^\sigma \frac{\theta^{2\sigma}}{4^\sigma} \geq \frac{1}{16}((1-r)^{2\sigma} + 2|\theta|^{2\sigma}),
\end{align}
which, combined with \eqref{eq:LHS}, proves \eqref{eq3.38-5} and thus \eqref{eq:f2_Holder} with $C=16$. The rest of the proof follows immediately from Lemma \ref{lem:alphanorm}.
\end{proof}

Define an operator $B$ as
\begin{equation}\label{e3.4}
B(v)=V v+f\left(|v|^2\right) v.
\end{equation}
For the operator $B$, from Lemma \ref{assum1}, we immediately have the following.
\begin{lemma}\label{lem3.4}
Let $v, w \in L^{\infty}(\Omega)$ satisfying $\|v\|_{L^{\infty}} \leq M_0$ and $\|w\|_{L^{\infty}} \leq M_0$, then
\begin{equation*}\label{e3.5}
\|B(v)-B(w)\|_{L^2} \leq C(M_0, \|V\|_{L^{\infty}})\|v-w\|_{L^2} .
\end{equation*}
\end{lemma}

\begin{lemma}\label{lem3.6}
Assume $\tau>0$ and $g \in C([0, \tau] ; L^2(\Omega)) \cap W^{1,1}([0, \tau] ; L^2(\Omega))$. If
\begin{equation}\label{e3.10}
w(t)=-i \int_0^t {e}^{i(t-s) \langle\nabla\rangle_{\alpha}} g(s) \mathrm{d} s, \quad t \in[0, \tau],
\end{equation}
then we have
\begin{equation}\label{e3.11}
\|\langle\nabla\rangle_{\alpha} w\|_{L^{\infty}\left([0, \tau] ; L^2\right)} \leq\|g\|_{L^{\infty}\left([0, \tau] ; L^2\right)}+\|g(0)\|_{L^2}+\left\|\partial_t g\right\|_{L^1\left([0, \tau] ; L^2\right)}.
\end{equation}
\end{lemma}
\begin{proof}
Apply the operator $\langle\nabla\rangle_{\alpha}$ to both sides of Eq.\eqref{e3.10}, take the Fourier transform and use integration by parts, we have
\begin{equation*}\label{e3.12n}
\begin{aligned}
&-|\mu_{l}|^{\alpha}\widehat{w}_{l}\\
&=i|\mu_{l}|^{\alpha}\int_0^t {e}^{-i(t-s)|\mu_{l}|^{\alpha}} \widehat{g}_{l}(s)\mathrm{d} s\\
&=e^{-it|\mu_{l}|^{\alpha}}\left[\widehat{g}_{l}(t)e^{it|\mu_{l}|^{\alpha}}-\widehat{g}_{l}(0)-\int_0^t{e}^{i s|\mu_{l}|^{\alpha}}\partial_{s}\widehat{g}_{l}(s)\mathrm{d} s\right].
\end{aligned}
\end{equation*}
Through the Parseval's identity and the property of $|e^{it|\mu_{l}|^{\alpha}}|=1$, we obtain
\begin{equation}\label{e3.14}
\|\langle\nabla\rangle_{\alpha} w(t)\|_{L^2} \leq\|g(0)\|_{L^2}+\left\|\partial_t g\right\|_{L^1\left([0, \tau] ; L^2\right)}+\|g(t)\|_{L^2},
\end{equation}
and the conclusion follows from taking supremum of $t$ on both sides.
\end{proof}
\begin{lemma}\label{lem3.9}
Let $v, w \in L^2(\Omega)$ and $0<\tau<1$. Then we have
$$
\left\|\varphi_1(i \tau \langle\nabla\rangle_{\alpha}) v-\varphi_1(i \tau \langle\nabla\rangle_{\alpha}) w\right\|_{H^\eta} \leq C(\eta) \tau^{-\frac{\eta}{\alpha}}\|v-w\|_{L^2}, \quad 0 \leq \eta \leq \alpha,
$$
where $C(\eta)=2^{\frac{\eta}{\alpha}}\left(1+\mu_1^{-2}\right)^{\frac{\eta}{2}}$.
\end{lemma}
\begin{proof}
It suffices to show that for any $v \in L^2(\Omega)$,
\begin{equation}\label{e3.22}
\left\|\varphi_1(i \tau \langle\nabla\rangle_{\alpha}) v\right\|_{H^\eta} \leq C(\eta) \tau^{-\frac{\eta}{\alpha}}\|v\|_{L^{2}}, \quad 0 \leq \eta \leq \alpha.
\end{equation}
Note that
\begin{equation}\label{e3.23}
\left|{e}^{i \rho}-1\right| \leq 2^\lambda \rho^{1-\lambda}, \quad \rho \in \mathbb{R}^{+}, \quad 0 \leq \lambda \leq 1 .
\end{equation}
By Parseval's identity, using \eqref{e3.23} with $\lambda=\frac{\eta}{\alpha}$ and recalling \eqref{e2.3}, we have
$$
\begin{aligned}
\|\varphi_1(i \tau \langle\nabla\rangle_{\alpha}) v\|_{H^\eta}^2 & =\sum_{l \in \mathbb{Z}}(1+\mu_l^2)^\eta|\varphi_1(-i \tau |\mu_l|^\alpha)|^2|\widehat{v}_l|^2 \\
& =|\widehat{v}_0|^2+\sum_{l \in \mathbb{Z} \backslash\{0\}}(1+\mu_l^2)^\eta\left|\frac{{e}^{i \tau |\mu_l|^\alpha}-1}{\tau |\mu_l|^\alpha}\right|^2|\widehat{v}_l|^2 \\
& \leq|\widehat{v}_0|^2+2^{\frac{2\eta}{\alpha}} \sum_{l \in \mathbb{Z} \backslash\{0\}}(1+\mu_l^2)^\eta(\tau |\mu_l|^\alpha)^{-\frac{2\eta}{\alpha}}|\widehat{v}_l|^2 \\
& =|\widehat{v}_0|^2+2^{\frac{2\eta}{\alpha}} \tau^{-\frac{2\eta}{\alpha}} \sum_{l \in \mathbb{Z} \backslash\{0\}}\left(\frac{1+\mu_l^2}{\mu_l^2}\right)^\eta|\widehat{v}_l|^2 \\
& \leq|\widehat{v}_0|^2+C(\eta)^2 \tau^{-\frac{2\eta}{\alpha}} \sum_{l \in \mathbb{Z} \backslash\{0\}}{|\widehat{v}_l|^2} \\
& \leq C(\eta)^2 \tau^{-\frac{2\eta}{\alpha}} \sum_{l \in \mathbb{Z}}{|\widehat{v}_l|^2}\\
&=C(\eta)^2 \tau^{-\frac{2\eta}{\alpha}} \|v\|_{L^2}^2,
\end{aligned}
$$
which proves \eqref{e3.22} and concludes the proof.
\end{proof}

\section{Optimal error bounds}
In this section, we establish optimal error bounds in $L^2$-norm and $H^{\frac{\alpha}{2}}$-norm for the semi-discretization \eqref{e2.2} and full discretization \eqref{e2.11} of the SFNLSE \eqref{e1.1}.

\subsection{Main results}
Let $T_\text{max}$ be the maximal existing time of the solution of the SFNLSE \eqref{e1.1} and take $0<T<T_\text{max}$ be a fixed time. Under the low regularity assumptions of potential and/or nonlinearity, i.e., $V \in L^\infty(\Omega)$ and $\sigma > 0$, we assume the exact solution $\psi \in C([0, T] ; H_{\text{\rm per}}^{\alpha}(\Omega)) \cap C^1([0, T] ; L^2(\Omega))$ and define
\begin{equation}\label{e3.1}
	M:=\max \left\{\|\psi\|_{L^{\infty}([0, T] ; H^{\alpha})},\|\psi\|_{L^{\infty}([0, T] ; L^{\infty})},\|\partial_t \psi\|_{L^{\infty}([0, T] ; L^2)},\|V\|_{L^{\infty}}\right\}.
\end{equation}
Note that our regularity assumption on the exact solution is compatible with the low regularity assumptions on potential and nonlinearity.

Let $\psi^{[n]}(0 \leq n \leq T / \tau)$ be the numerical approximation obtained by the EWI \eqref{e2.2}, then we have
\begin{theorem}[Optimal error bounds for the semi-discretization]\label{th3.1}
Under the assumptions that $V \in L^{\infty}(\Omega)$, $\sigma>0$, and $\psi \in C([0, T] ; H_{\text{\rm per}}^{\alpha}(\Omega)) \cap C^1([0, T] ; L^2(\Omega))$, there exists $\tau_0>0$ depending on $M$ and $T$ and sufficiently small such that for any $0<\tau<\tau_0$, we have $\psi^{[n]} \in H_{\text{\rm per}}^{\alpha}(\Omega)$ for $0 \leq n \leq T / \tau$ and
\begin{equation}\label{e3.2}
\begin{aligned}
&\|\psi(\cdot, t_n)-\psi^{[n]}\|_{L^2} \lesssim \tau, \quad\|\psi^{[n]}\|_{H^{\alpha}} \leq C(M), \\
&\|\psi(\cdot, t_n)-\psi^{[n]}\|_{H^\frac{\alpha}{2}} \lesssim \sqrt{\tau}, \quad 0 \leq n \leq \frac{T}{\tau}.
\end{aligned}
\end{equation}
Moreover, if $V \in H_{\mathrm{per}}^{\frac{\alpha}{2}}(\Omega)$, $\sigma>\frac{\alpha}{4}$, and $\psi \in C([0, T] ; H_{\mathrm{per}}^{\frac{3\alpha}{2}}(\Omega)) \cap C^1([0, T] ; H^{\frac{\alpha}{2}}(\Omega))$, we have, for $0<\tau<\tau_0$,
\begin{equation}\label{e3.3}
\|\psi(\cdot, t_n)-\psi^{[n]}\|_{H^\frac{\alpha}{2}} \lesssim \tau, \quad 0 \leq n \leq \frac{T}{\tau}.
\end{equation}
\end{theorem}

For $\psi^n(0 \leq n \leq T / \tau)$ obtained by the EWI-FS scheme \eqref{e2.11}, we have
\begin{theorem}[Optimal error bounds for the full discretization]\label{th4.1}
Under the assumptions that $V \in L^{\infty}(\Omega)$, $\sigma >0$, and $\psi \in C([0, T] ; H_{\mathrm{per}}^{m}(\Omega)) \cap C^1([0, T] ; L^2(\Omega))$ for some $m\geq \alpha$, there exists $\tau_0>0$ and $h_0>0$ depending on $M$ and $T$ and sufficiently small such that for any $0<\tau<\tau_0$ and $0<h<h_0$, we have
\begin{equation}\label{e4.1}
\begin{aligned}
& \|\psi(\cdot, t_n)-\psi^n\|_{L^2} \lesssim \tau+h^{m}, \quad\|\psi^n\|_{H^\alpha} \leq C(M), \\
& \|\psi(\cdot, t_n)-\psi^n\|_{H^\frac{\alpha}{2}} \lesssim \sqrt{\tau}+h^{m-\frac{\alpha}{2}}, \quad 0 \leq n \leq \frac{T}{\tau}.
\end{aligned}
\end{equation}
Moreover, if $V \in H_{\mathrm{per}}^{\frac{\alpha}{2}}(\Omega)$, $\sigma > \frac{\alpha}{4}$, and $\psi \in C([0, T] ; H_{\mathrm{per}}^{m}(\Omega)) \cap$ $C^1([0, T] ; H^{\frac{\alpha}{2}}(\Omega))$, with $m\geq \frac{3\alpha}{2}$, for $0<\tau<\tau_0$ and $0<h<h_0$, we have
\begin{equation}\label{e4.2}
\|\psi(\cdot, t_n)-\psi^n\|_{L^2} \lesssim \tau+h^{m}, \quad \|\psi(\cdot, t_n)-\psi^n\|_{H^{\frac{\alpha}{2}}} \lesssim \tau+h^{m-\frac{\alpha}{2}}, \  0 \leq n \leq \frac{T}{\tau}.
\end{equation}
\end{theorem}

\begin{remark}
	Compared to Theorem \ref{th3.1}, we assume the exact solution has possibly higher regularity $H^m$ with $m \geq \alpha $ for \eqref{e4.1} and $m \geq {3\alpha}/{2}$ for \eqref{e4.2} to show that the spatial convergence order of the FS method is always optimal with respect to the actual regularity of the exact solution. This is confirmed by extensive numerical results in Section\ref{sec:num}.
\end{remark}

\begin{remark}[Generalization to higher dimensions]\label{rem4.4}
Both Theorem \ref{th3.1} and \ref{th4.1} are stated for the 1D case. For the 2D case, \eqref{e3.2} and \eqref{e4.1} remains unchanged while, for \eqref{e3.3} and \eqref{e4.2} to hold, we need to assume $\sigma \geq 1/2$. For the 3D case, we need, $\alpha >3/2$ for \eqref{e3.2} and \eqref{e4.1} to ensure $H^\alpha \hookrightarrow L^\infty$, and, need, in addition, $\sigma >3/4$ for \eqref{e3.3} and \eqref{e4.2}. The details of the proof can be found in the Appendix.
\end{remark}

In the following, we first prove Theorem \ref{th3.1}. We start with the proof of \eqref{e3.2}, and the proof of \eqref{e3.3} can be obtained by the discrete Gronwall's inequality with the established uniform $H^{\alpha}$-bound of the semi-discretization solution in \eqref{e3.2}.

\subsection{Proof of Theorem \ref{th3.1}}
Before proving Theorem \ref{th3.1}, we first present two important propositions.
\begin{proposition}[Local truncation error in semi-discretization]\label{prop3.7} For $0 \leq n \leq T / \tau-1$, we have
\begin{equation}\label{e3.15}
\|\psi(t_{n+1})-\Phi^\tau(\psi(t_n))\|_{H^\eta} \lesssim \tau^{2-\frac{\eta}{\alpha}}, \quad 0 \leq \eta \leq \alpha.
\end{equation}
\end{proposition}
\begin{proof}
Recalling \eqref{e2.1}, \eqref{e2.3-1} and \eqref{e3.4}, we have
\begin{align}\label{e3.18}
\psi(t_{n+1})-\Phi^\tau(\psi(t_n)) & =-i \int_0^\tau {e}^{i(\tau-s) \langle\nabla\rangle_{\alpha}}(B(\psi(t_n+s))-B(\psi(t_n))) \mathrm{d} s \notag\\
& =-i \int_0^\tau {e}^{i(\tau-s) \langle\nabla\rangle_{\alpha}} g_n(s) \mathrm{d} s, \quad 0 \leq n \leq \frac{T}{\tau}-1,
\end{align}
where $g_n(t):=B(\psi(t_n+t))-B(\psi(t))$ for $0 \leq s \leq \tau$. By Lemma 3.5 in \cite{Bao2024}, $g \in C([0, \tau] ; L^2(\Omega)) \cap W^{1, \infty}([0, \tau] ; L^2(\Omega))$ satisfies
\begin{align}
	&\left\|g\right\|_{L^{\infty}\left([0, \tau] ; L^2\right)} \lesssim \tau, \label{e3.7} \\
	&\left\|\partial_t g\right\|_{L^{\infty}\left([0, \tau] ; L^2\right)} \lesssim 1, \label{e3.8}
\end{align}
From \eqref{e3.18}, using \eqref{e3.7}, one gets
\begin{equation}\label{e3.19}
\|\psi(t_{n+1})-\Phi^\tau(\psi(t_n))\|_{L^2} \leq \int_0^\tau\|g_n(s)\|_{L^2}\mathrm{d} s \lesssim \tau^2,
\end{equation}
which proves \eqref{e3.15} for $\eta=0$. Then we will prove \eqref{e3.15} with $\eta=\alpha$. Applying Lemma \ref{lem3.6} to \eqref{e3.18}, using \eqref{e3.7}, \eqref{e3.8} and noting $g_n(0)=0$, we have
\begin{align}\label{e3.20}
& \|\langle\nabla\rangle_{\alpha}(\psi(t_{n+1})-\Phi^\tau(\psi(t_n)))\|_{L^2} \notag\\
& \leq\|g_n\|_{L^{\infty}([0, \tau] ; L^2)}+\|g_n(0)\|_{L^2}+\|\partial_t g_n\|_{L^1([0, \tau] ; L^2)} \notag\\
& \lesssim  \tau+\tau\|\partial_t g_n\|_{L^{\infty}([0, \tau] ; L^2)} \lesssim  \tau,
\end{align}
which combined with \eqref{e3.19} implies
\begin{equation}\label{e3.21}
\|\psi(t_{n+1})-\Phi^\tau(\psi(t_n))\|_{H^\alpha} \lesssim \tau, \quad 0 \leq n \leq \frac{T}{\tau}-1.
\end{equation}
The conclusion \eqref{e3.15} with $0<\eta<\alpha$ will follow from the Gagliardo-Nirenberg interpolation inequalities.
\end{proof}

\begin{proposition}[Stability estimate of semi-discretization]\label{prop3.10} Let $v, w \in H_{\mathrm{per}}^{\alpha}(\Omega)$ such that $\|v\|_{L^{\infty}} \leq$ $M_0$ and $\|w\|_{L^{\infty}} \leq M_0$ and let $0<\tau<1$. Then we have, for $0 \leq \eta \leq \alpha$,
$$
\|\Phi^\tau(v)-\Phi^\tau(w)\|_{H^\eta} \leq\|v-w\|_{H^\eta}+C(M_0) \tau^{1-\frac{\eta}{\alpha}}\|v-w\|_{L^2}.
$$
\end{proposition}

\begin{proof}
Recalling \eqref{e2.2} and \eqref{e3.4}, we have
\begin{equation}\label{e3.24}
\Phi^\tau(u)={e}^{i \tau \langle\nabla\rangle_{\alpha}} u-i \tau \varphi_1(i \tau \langle\nabla\rangle_{\alpha}) B(u), \quad u \in H_{\mathrm{per}}^{\alpha}(\Omega) .
\end{equation}

Using the isometry property of ${e}^{i t \langle\nabla\rangle_{\alpha}}$, Lemma \ref{lem3.9} and Lemma \ref{lem3.4}, we have
\begin{align}
&\|\Phi^\tau(v)-\Phi^\tau(w)\|_{H^\eta} \notag \\
& \leq \|{e}^{i \tau \langle\nabla\rangle_{\alpha}} v-{e}^{i \tau \langle\nabla\rangle_{\alpha}} w\|_{H^\eta}+\tau\|\varphi_1(i \tau \langle\nabla\rangle_{\alpha})(B(v)-B(w))\|_{H^\eta} \notag \\
& \leq\|v-w\|_{H^\eta}+C(\eta) \tau^{1-\frac{\eta}{\alpha}}\|B(v)-B(w)\|_{L^2} \notag \\
& \leq\|v-w\|_{H^\eta}+C(\eta) \tau^{1-\frac{\eta}{\alpha}} C(M_0,\|V\|_{L^{\infty}})\|v-w\|_{L^2} .
\end{align}

The conclusion follows from letting $C(M_0)=C(\eta) C(M_0,\|V\|_{L^{\infty}})$.

\end{proof}

\begin{proof}[Proof of Theorem \ref{th3.1}.]
Define the error function $e^{[n]}:=\psi\left(t_n\right)-\psi^{[n]}$ for $0 \leq n \leq T / \tau$. For $0 \leq n \leq T / \tau-1$ and $0 \leq \eta \leq \alpha$, we have
\begin{equation}\label{e3.25}
\begin{aligned}
\|e^{[n+1]}\|_{H^\eta} & =\|\psi(t_{n+1})-\psi^{[n+1]}\|_{H^\eta}=\|\psi(t_{n+1})-
\Phi^\tau(\psi^{[n]})\|_{H^\eta} \\
& \leq\|\psi(t_{n+1})-\Phi^\tau(\psi(t_n))\|_{H^\eta}+
\|\Phi^\tau(\psi(t_n))-\Phi^\tau(\psi^{[n]})\|_{H^\eta} .
\end{aligned}
\end{equation}

We will show that when $0<\tau<\tau_{0}<1$, for $0\leq n\leq T/\tau$
\begin{equation}\label{e3.27}
\|e^{[n]}\|_{L^2} \lesssim \tau, \quad\|e^{[n]}\|_{H^{\varrho\alpha}} \lesssim \tau^{1-\varrho},
\end{equation}
where $0<\varrho<1, \varrho\alpha>\frac{d}{2}=\frac{1}{2}$.

We shall use an induction argument to prove it. When $n=0$, $e^{[n]}=0$, \eqref{e3.27} holds. We assume that \eqref{e3.27} holds for $0\leq n\leq m\leq T/\tau-1$. Taking $\eta=0$ and $\eta=\varrho\alpha$ in \eqref{e3.25}, using Proposition \ref{prop3.7} and \ref{prop3.10}, we obtain,
\begin{equation}\label{e3.31}
\begin{aligned}
\|e^{[n+1]}\|_{L^2} \leq(1+C_0 \tau)\|e^{[n]}\|_{L^2}+C_1 \tau^2,
\end{aligned}
\end{equation}
\begin{equation}\label{e3.32}
\begin{aligned}
\|e^{[n+1]}\|_{H^{\varrho\alpha}} \leq \|e^{[n]}\|_{H^{\varrho\alpha}}+C_0 \tau^{1-\varrho}\|e^{[n]}\|_{L^2}+C_1 \tau^{2-\varrho},
\end{aligned}
\end{equation}
under the given regularity assumptions and the assumption for induction, by the fractional Sobolev embedding $H^{\varrho\alpha} \hookrightarrow L^{\infty}$ (see Theorem 2.2.9 (5) in \cite{Guo2011}), $C_0$ and $C_1$ are uniformly bounded. Applying the discrete Gronwall's inequality to \eqref{e3.31}, we have $\|e^{[m+1]}\|_{L^2} \lesssim \tau$. Summing over $n$ from 0 to $m$ in \eqref{e3.32}, we obtain $\|e^{[m+1]}\|_{H^{\varrho\alpha}}\lesssim \tau^{1-\varrho}$. Then we prove \eqref{e3.27} for $n=m+1$, and thus for all $0 \leq n \leq T / \tau$ by mathematical induction. The rest of the proof follows immediately. More similar details can be found in \citep{Bao2024,Bao2024sEWI}.

To prove \eqref{e3.3} in Theorem \ref{th3.1}, we assume that $V \in  H_{\text {per}}^{\frac{\alpha}{2}}(\Omega)$, $\sigma > \frac{\alpha}{4}$, $\psi \in C([0, T] ; H_{\text {per}}^{\frac{3\alpha}{2}}(\Omega)) \cap$ $C^1([0, T] ; H^{\frac{\alpha}{2}}(\Omega))$ and $0<\tau<\tau_0$. By Lemma \ref{assum2}, we know $B: H_{\text {per }}^{\frac{\alpha}{2}}(\Omega) \rightarrow H_{\text {per }}^{\frac{\alpha}{2}}(\Omega)$ satisfies
\begin{equation}\label{e3.38}
\|B(v)-B(w)\|_{H^\frac{\alpha}{2}} \leq C(\|v\|_{H^{1}}^{2\sigma},\|w\|_{H^{1}}^{2\sigma},\|V\|_{H^{\frac{\alpha}{2}}})\|v-w\|_{H^\frac{\alpha}{2}}.
\end{equation}

Using \eqref{e3.38}, the isometry property of ${e}^{i t \langle\nabla\rangle_{\alpha}}$, and noting that $\psi \in C^1([0, T] ; H^{\frac{\alpha}{2}}(\Omega))$, we have, for $0 \leq n \leq T / \tau-1$,
\begin{equation}\label{e3.39}
\|\psi(t_{n+1})-\Phi^\tau(\psi(t_n))\|_{H^{\frac{\alpha}{2}}} \leq \int_0^\tau\|B(\psi(t_n+s))-B(\psi(t_n))\|_{H^{\frac{\alpha}{2}}} \mathrm{d} s \lesssim \tau^2 .
\end{equation}
Noting that $|\varphi_1(i \theta)| \leq 1$ for $\theta \in \mathbb{R}$, we have
\begin{equation}\label{e3.40}
\|\varphi_1(i \tau \Delta) v\|_{H^{\frac{\alpha}{2}}} \leq\|v\|_{H^{\frac{\alpha}{2}}}, \quad v \in H_{\mathrm{per}}^{\frac{\alpha}{2}}(\Omega),
\end{equation}
which implies, by recalling \eqref{e3.24} and using \eqref{e3.38} again,
\begin{align}\label{e3.41}
\|\Phi^\tau(\psi(t_n))-\Phi^\tau(\psi^{[n]})\|_{H^{\frac{\alpha}{2}}} & \leq\|\psi(t_n)-\psi^{[n]}\|_{H^{\alpha/2}}+\tau\|B(\psi(t_n))
-B(\psi^{[n]})\|_{H^{\frac{\alpha}{2}}} \notag\\
& \leq(1+C_{2} \tau)\|\psi(t_n)-\psi^{[n]}\|_{H^{\frac{\alpha}{2}}}, \quad 0 \leq n \leq \frac{T}{\tau}-1,
\end{align}
where $C_{2}$ depends on $\|V\|_{H^{\frac{\alpha}{2}}},\|\psi(t_n)\|_{H^{1}}^{2\sigma}$ and $\|\psi^{[n]}\|_{H^{1}}^{2\sigma}$, which are uniformly bounded. Then \eqref{e3.3} follows from \eqref{e3.39} and \eqref{e3.41} by the discrete Gronwall's inequality.
\end{proof}

\subsection{Proof of Theorem \ref{th4.1}}
In this subsection, we prove Theorem \ref{th4.1} for the full discretization by the EWI-FS method. Particular attention is paid to avoiding any coupling condition between the time step size $\tau$ and the mesh size $h$, and to establishing optimal spatial convergence orders with respect to the regularity of the exact solution. One should note that the technique used in \citep{Bao2024}, i.e., comparing the semi-discrete numeircal solution with the fully discrete numerical solution, cannot give optimal orders when $m > \alpha$ in \eqref{e4.1} since we only have a uniform $H^\alpha$-norm bound of the semi-discrete numerical solution. Here, we directly estimate the error between the exact solution and the fully discrete numerical solution with techniques from Theorem 6.1 in \citep{wangthesis}.

\begin{proposition}[Local truncation error of the full discretization]\label{prop4.1-1} Under the assumptions of Theorem \ref{th4.1}, for $0 \leq n \leq T / \tau-1$, we have
\begin{equation}\label{e3.15-1}
\|P_{N}\psi(t_{n+1})-\Phi_{h}^\tau(P_{N}\psi(t_n))\|_{H^\eta} \lesssim \tau^{2-\frac{\eta}{\alpha}}+ \tau h^{m-\eta}, \quad 0 \leq \eta \leq \alpha.
\end{equation}
\end{proposition}
\begin{proof}
Recalling \eqref{e2.1}, \eqref{e2.11} and \eqref{e3.4}, we get
\begin{align}\label{eqq4.10}
	&P_{N}\psi(t_{n+1})-\Phi_{h}^\tau(P_{N}\psi(t_n)) \notag \\
	&=-i \int_0^\tau {e}^{i(\tau-s) \langle\nabla\rangle_{\alpha}}P_{N}(B(\psi(t_n+s))-B(P_{N}\psi(t_n))) \mathrm{d} s \notag \\
	& =-i \int_0^\tau {e}^{i(\tau-s) \langle\nabla\rangle_{\alpha}} [P_{N}(B(\psi(t_n+s))-B(\psi(t_n))) \notag \\
	&\hspace{9em} +P_{N}(B(\psi(t_n)))-P_{N}(B(P_{N}\psi(t_n)))] \mathrm{d} s.
\end{align}
By the isometry property of $e^{i t \langle\nabla\rangle_{\alpha}}$, the standard projection error estimate of $P_{N}$, $\|\phi-P_N \phi\|_{L^2} \lesssim h^{m}\|\phi\|_{H^{m}}, \forall \phi \in H_{\text {per }}^{m}(\Omega)$  and $\|P_N \phi\|_{H^{\eta}} \leq \|\phi\|_{H^{\eta}}$, the inverse estimate $\|\phi\|_{H^{\eta}}\lesssim h^{-\eta}\|\phi\|_{L^{2}}, \forall \phi \in Y_{N}$, Proposition \ref{prop3.7} and Lemma \ref{lem3.4}, we have
\begin{align}\label{eqq4.11}
&\|P_{N}\psi(t_{n+1})-\Phi_{h}^\tau(P_{N}\psi(t_n))\|_{H^{\eta}}\notag\\
&\leq \|\psi(t_{n+1})-\Phi^\tau(\psi(t_n))\|_{H^\eta}+\tau\|B(\psi(t_n))-B(P_N\psi(t_n))\|_{H^\eta}\notag\\
&\lesssim \tau^{2-\frac{\eta}{\alpha}}+\tau h^{-\eta}\|\psi(t_n)-P_N\psi(t_n)\|_{L^2}\notag\\
&\lesssim \tau^{2-\frac{\eta}{\alpha}}+\tau h^{m-\eta},
\end{align}
which concludes the proof.
\end{proof}

Following the proof of Lemma \ref{lem3.4}, \ref{lem3.9} and Proposition \ref{prop3.10}, we have the following stability estimate.
\begin{proposition}[Stability estimate of the full discretization]\label{prop4.2-1} Let $v, w \in Y_{N}$, such that $\|v\|_{L^{\infty}}\leq M_0$ and $\|w\|_{L^{\infty}}\leq M_0$. For any $\tau>0$, $0 \leq \eta \leq \alpha$, we have
\begin{equation}\label{e3.16-1}
\|\Phi_{h}^{\tau}(v)-\Phi_{h}^{\tau}(w)\|_{H^{\eta}}\leq \|v-w\|_{H^\eta}+C(M_0) \tau^{1-\frac{\eta}{\alpha}}\|v-w\|_{L^2}.
\end{equation}
\end{proposition}

\begin{proof}[Proof of Theorem \ref{th4.1}.]
By triangle inequality,
\begin{equation}\label{eq3.16-2}
	\| \psi(t_n)-\psi^n \|_{H^\alpha} \leq \|\psi(t_n)-\psi^{[n]}\|_{H^\alpha}+\|\psi^{[n]}-P_{N}\psi^{[n]}\|_{H^\alpha}+\|P_{N}\psi^{[n]}-\psi^n\|_{H^\alpha}.
\end{equation}
From \eqref{e3.2}, using the interpolation inequality, we have
\begin{equation}\label{eq3.16-3}
 \|\psi(t_n)-\psi^{[n]}\|_{H^\alpha}\lesssim 1.
\end{equation}
Following a similar procedure as Proposition 4.3 in \citep{Bao2024}, Theorem \ref{th3.1} and the inverse inequality $\|\phi\|_{H^{\alpha}}\lesssim h^{-\alpha}\|\phi\|_{L^2}$ for $\phi\in Y_{N}$ \citep{Canuto2006}, we have a uniform $H^\alpha$-norm bound of the fully discrete solution:
\begin{equation}
	\| \psi^n \|_{H^\alpha} \leq C(M), \quad 0 \leq n \leq \frac{T}{\tau},
\end{equation}
which implies, by the Sobolev embedding $H^{\alpha} \hookrightarrow L^\infty$, $\| \psi^n \|_{L^\infty} \ (0 \leq n \leq T/\tau)$ is also uniformly bounded for $0<\tau\leq \tau_0$.

Define the error function $e^{n}:=P_N \psi(t_{n})-\psi^{n}$ for $0\leq n\leq T/\tau$, then $e^0=0$. For $0\leq n\leq T/\tau-1$,
\begin{align}\label{eqq4.12}
e^{n+1}&=P_N \psi(t_{n+1})-\psi^{n+1}\notag\\
&=P_N \psi(t_{n+1})-\Phi_{h}^{\tau}(P_{N}\psi(t_{n}))+\Phi_{h}^{\tau}(P_{N}\psi(t_{n}))-\Phi_{h}^{\tau}(\psi^{n}).
\end{align}
From \eqref{eqq4.12}, using Proposition \ref{prop4.1-1} and \ref{prop4.2-1} with $\eta=0$, we have
\begin{equation}\label{eqq4.13}
\begin{aligned}
\|e^{n+1}\|_{L^2}\leq (1+C_3\tau)\|e^{n}\|_{L^2}+C_4\tau(\tau+h^m), \ 0\leq n\leq \frac{T}{\tau}-1,
\end{aligned}
\end{equation}
where $C_3$ depends on $M$ and $\|\psi^{n}\|_{L^{\infty}}$, and $C_4$ depends on $M$. Thus, both $C_3$ and $C_4$ are uniformly bounded. By discrete Gronwall's inequality, we have
\begin{equation}\label{eqq4.14}
\|e^{n}\|_{L^2}\lesssim \tau+h^m, \ 0\leq n \leq \frac{T}{\tau}.
\end{equation}

To obtain error estimates in $H^{\frac{\alpha}{2}}-$norm with optimal spatial convergence orders, we separately consider two cases: (i) $\tau\geq h^{\alpha}$ and (ii) $\tau<h^{\alpha}$. We start with (i), from \eqref{eqq4.12}, by Proposition \ref{prop4.1-1} and \ref{prop4.2-1} with $\eta=\alpha/2$ and \eqref{eqq4.14}, we have
\begin{align}\label{eqq4.15}
\|e^{n+1}\|_{H^{\frac{\alpha}{2}}}&\leq \|e^{n}\|_{H^{\frac{\alpha}{2}}}+C_3\tau^{\frac{1}{2}}\|e^n\|_{L^2}+C_4\tau(\tau^{\frac{1}{2}}+h^{m-\frac{\alpha}{2}})\notag\\
&\leq \|e^{n}\|_{H^{\frac{\alpha}{2}}}+C_5\tau^{\frac{1}{2}}(\tau+h^m)+C_4\tau(\tau^{\frac{1}{2}}+h^{m-\frac{\alpha}{2}}),
\end{align}
which implies by summing from 0 to $n$ and using $\tau\geq h^{\alpha}$ that
\begin{align}\label{eqq4.16}
\|e^{n+1}\|_{H^{\frac{\alpha}{2}}}
&\leq C_5(n+1)\tau^{\frac{1}{2}}(\tau+h^m)+C_4(n+1)\tau(\tau^{\frac{1}{2}}+h^{m-\frac{\alpha}{2}})\notag\\
&\lesssim T\tau^{-\frac{1}{2}}(\tau+h^m)+T(\tau^{\frac{1}{2}}+h^{m-\frac{\alpha}{2}})\notag\\
&\lesssim \tau^{\frac{1}{2}}+h^{m-\frac{\alpha}{2}}\left(\frac{h^{\frac{\alpha}{2}}}{\tau^{\frac{1}{2}}}\right)+\tau^{\frac{1}{2}}+h^{m-\frac{\alpha}{2}}\notag\\
&\lesssim \tau^{\frac{1}{2}}+h^{m-\frac{\alpha}{2}}.
\end{align}
For case (ii), the error bound can be obtained by the inverse inequality as
\begin{align}\label{eqq4.17}
\|e^{n+1}\|_{H^{\frac{\alpha}{2}}}&\lesssim h^{-\frac{\alpha}{2}}\|e^{n+1}\|_{L^2}\lesssim h^{-\frac{\alpha}{2}}(\tau+h^m)\notag\\
&=\tau^{\frac{1}{2}}\left(\frac{\tau^{\frac{1}{2}}}{h^{\frac{\alpha}{2}}}\right)+h^{m-\frac{\alpha}{2}}\leq \tau^{\frac{1}{2}}+h^{m-\frac{\alpha}{2}}.
\end{align}
Combining \eqref{eqq4.16} and \eqref{eqq4.17}, we have
\begin{equation}\label{eqq4.18}
\begin{aligned}
\|e^{n+1}\|_{H^{\frac{\alpha}{2}}}\lesssim \tau^{\frac{1}{2}}+h^{m-\frac{\alpha}{2}},\quad 0\leq n\leq \frac{T}{\tau}.
\end{aligned}
\end{equation}
Together with \eqref{eqq4.14}, and $\psi(\cdot, t_{n})-\psi^{n}=\psi(\cdot, t_{n})-P_N \psi(t_{n})+P_N \psi(t_{n})-\psi^{n}=\psi(\cdot, t_{n})-P_N \psi(t_{n})+e^n$, using the standard projection error estimate, we complete the proof of \eqref{e4.1} in Theorem \ref{th4.1}.

To prove \eqref{e4.2}, we assume that $V \in H_{\text {per }}^{\frac{\alpha}{2}}(\Omega)$, $\sigma > \frac{\alpha}{4}$, $\psi \in C([0, T] ; H_{\text {per}}^{m}(\Omega)) \cap$ $C^1([0, T] ; H^{\frac{\alpha}{2}}(\Omega))$ with $m\geq \frac{3\alpha}{2}$ and $0<\tau<\tau_0, 0<h<h_0$.  Since we already have the uniform control of $\psi^n$ in $H^\alpha$-norm, \eqref{e4.2} can be proved similarly to \eqref{e4.1}, and we just sketch the outline here.

Using Lemma \ref{lem3.4} and \eqref{e3.38}, we have, for $0 \leq n \leq T / \tau-1$, the local truncation error is
\begin{equation}\label{e4.22}
\begin{aligned}
& \|P_N \psi(t_{n+1})-\Phi_h^\tau(P_N \psi(t_n))\|_{L^2} \lesssim \tau^2+\tau h^{m}, \\
& \|P_N \psi(t_{n+1})-\Phi_h^\tau(P_N \psi(t_n))\|_{H^{\frac{\alpha}{2}}} \lesssim \tau^2+\tau h^{m-\frac{\alpha}{2}}. \\
\end{aligned}
\end{equation}
Besides, recalling \eqref{e2.11}, using Lemmas \ref{lem3.4} and \ref{lem3.9}, \eqref{e3.38} and \eqref{e3.40}, we obtain the stability estimates
\begin{equation}\label{e4.23}
\begin{aligned}
& \|\Phi_h^\tau(P_N \psi(t_n))-\Phi_h^\tau(\psi^n)\|_{L^2} \leq(1+C_6 \tau)\|P_N \psi(t_n)-\psi^n\|_{L^2}, \\
& \|\Phi_h^\tau(P_N \psi(t_n))-\Phi_h^\tau(\psi^n)\|_{H^{\frac{\alpha}{2}}} \leq(1+C_7 \tau)\|P_N \psi(t_n)-\psi^n\|_{H^{\frac{\alpha}{2}}},
\end{aligned}
\end{equation}
where $C_6$ depends on $\|P_N \psi(t_n)\|_{L^{\infty}}$ and $\|\psi^n\|_{L^{\infty}}$, and $C_7$ depends on $\|P_N \psi(t_n)\|_{H^{1}}^{2\sigma}$ and $\|\psi^n\|_{H^{1}}^{2\sigma}$. Thus, both $C_6$ and $C_7$ are under control. Then the proof can be completed by the discrete Gronwall's inequality and the standard projection error estimates of $P_N$.

\end{proof}

\section{Numerical results}\label{sec:num}
In this section, we present some numerical examples for the SFNLSE with either low regularity potential or nonlinearity. We choose $\Omega=(-16,16), T=1, d=1$ and consider the power-type nonlinearity $f(\rho)=$ $-\rho^\sigma(\sigma>0)$.

Let $\psi^n(0 \leq n \leq T / \tau)$ be the numerical solution obtained by the EWI-FS method or the EWI-FP method, which will be made clear in each case. Define the error functions
$$
e_{L^2}(t_n)=\|\psi(t_n)- \psi^n\|_{L^2}, \quad e_{H^\frac{\alpha}{2}}(t_n)=\|\psi(t_n)- \psi^n \|_{H^\frac{\alpha}{2}}, \quad 0 \leq n \leq \frac{T}{\tau}.
$$

\subsection{The SFNLSE with low regularity potential}\label{sub5.1new}
In this subsection, we consider the cubic SFNLSE with low regularity potential as
\begin{equation}\label{e5.1n}
\begin{aligned}
\quad i \partial_t \psi(x, t)&=(-\Delta)^{\frac{\alpha}{2}} \psi(x, t)+V(x) \psi(x, t)-|\psi(x, t)|^{2} \psi(x, t), \\
\psi_0(x)&=e^{-\frac{x^2}{2}},
\end{aligned}
\end{equation}
where $x \in \Omega, t>0$. $V(x)$ is chosen as the following two types:\\
(i) Type $\textrm{\uppercase\expandafter{\romannumeral1}}$.
\begin{equation}\label{e5.2n}
V_1(x)=\left\{\begin{array}{lr}
-4, & x \in(-2,2) \\
0, & \text { otherwise }
\end{array}.\right.
\end{equation}
(ii) Type $\textrm{\uppercase\expandafter{\romannumeral2}}$.
\begin{equation}\label{e5.3n}
V_2(x)=\textrm{real}\left(\sum_{l\in T_{M}}\widehat{v}_{l}e^{i\mu_{l}(x-a)}\right), \quad x \in \Omega,
\end{equation}
where
\begin{equation}\label{e5.3nn}
\widehat{v}_{l}=\left\{\begin{array}{ll}
\frac{x_{l}}{|\mu_{l}|^{\frac{1}{2}+0.01}}, & l\in T_{M}\backslash\{0\}, \quad M=2^{18}, \\
1, & l=0,
\end{array}\right.
\end{equation}
$x_l=\operatorname{rand}(-\frac{1}{2}, \frac{1}{2})+i \operatorname{rand}(-\frac{1}{2},\frac{1}{2}).$

\begin{figure}[ht!]
\centering
\includegraphics[scale=.42]{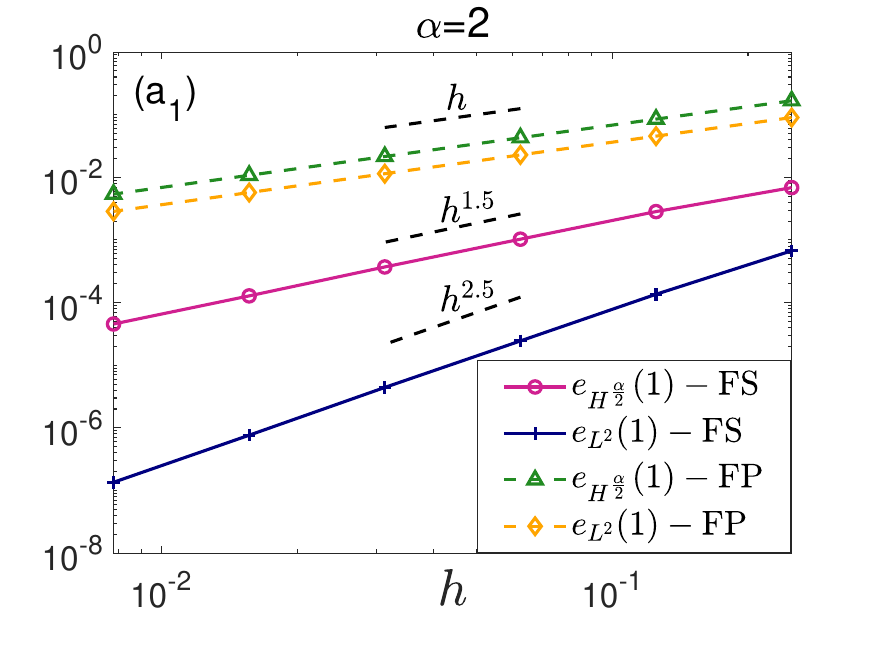}
\includegraphics[scale=.42]{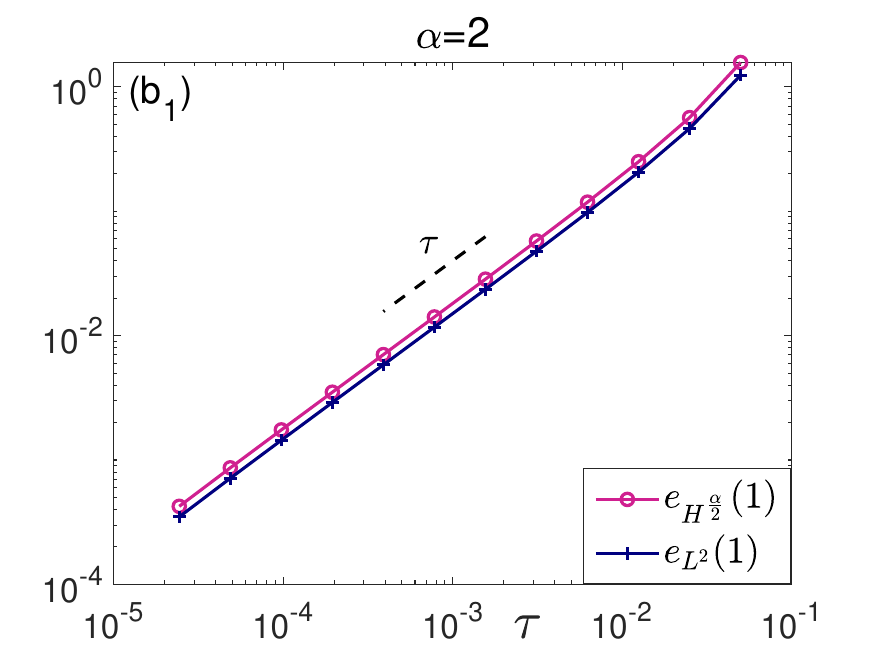}
\includegraphics[scale=.42]{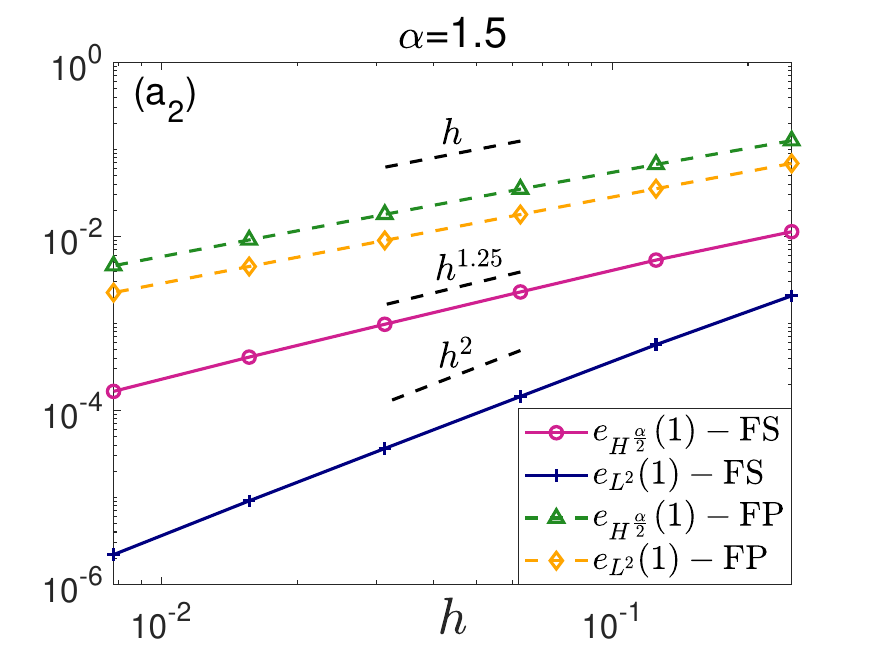}
\includegraphics[scale=.42]{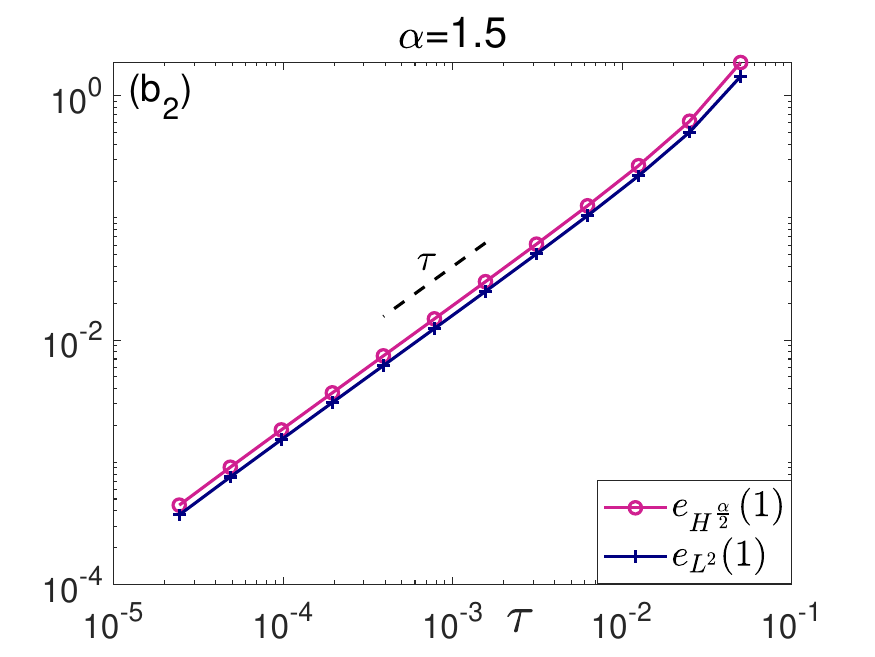}
\includegraphics[scale=.42]{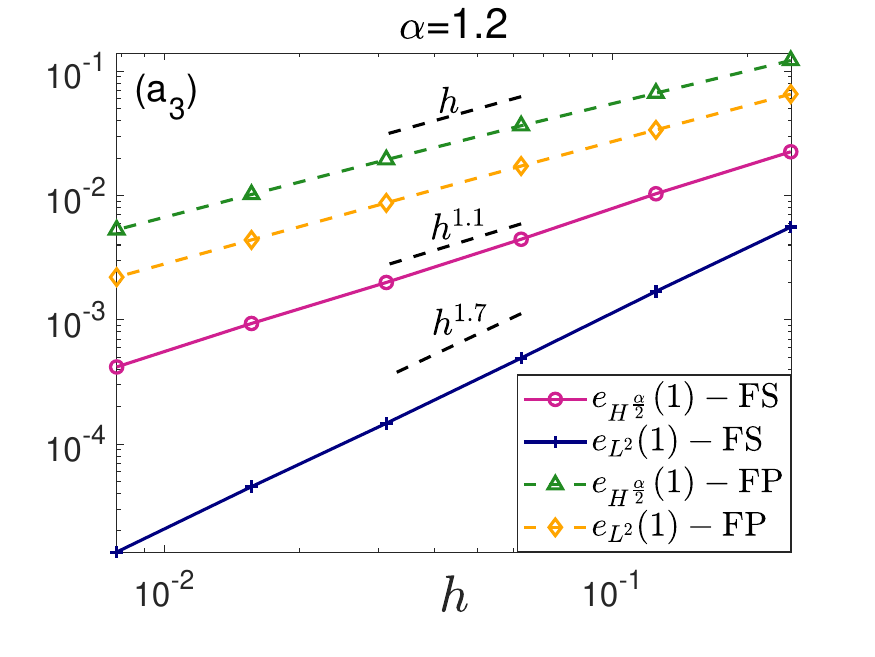}
\includegraphics[scale=.42]{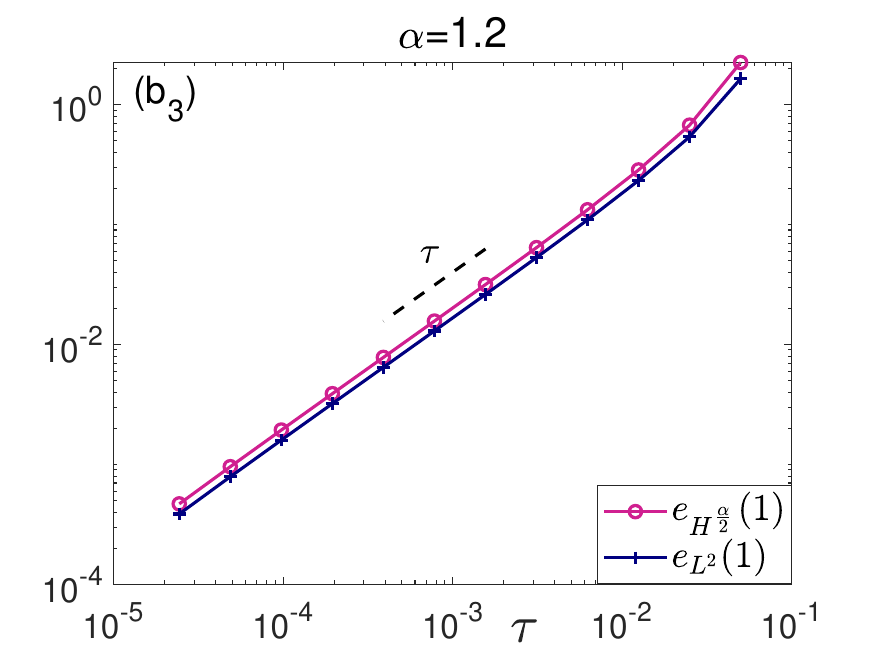}
\caption{Convergence tests of the EWI for \eqref{e5.1n} with low regularity potential $V_1$ and smooth initial condition $\psi_0$ : $(a_{1})-(a_{3})$ spatial errors of the FS and FP discretizations, and $(b_{1})-(b_{3})$ temporal errors in $L^2$-norm and $H^{\frac{\alpha}{2}}$-norm.}
\label{fig1}
\end{figure}

\begin{figure}[ht!]
\centering
\includegraphics[scale=.42]{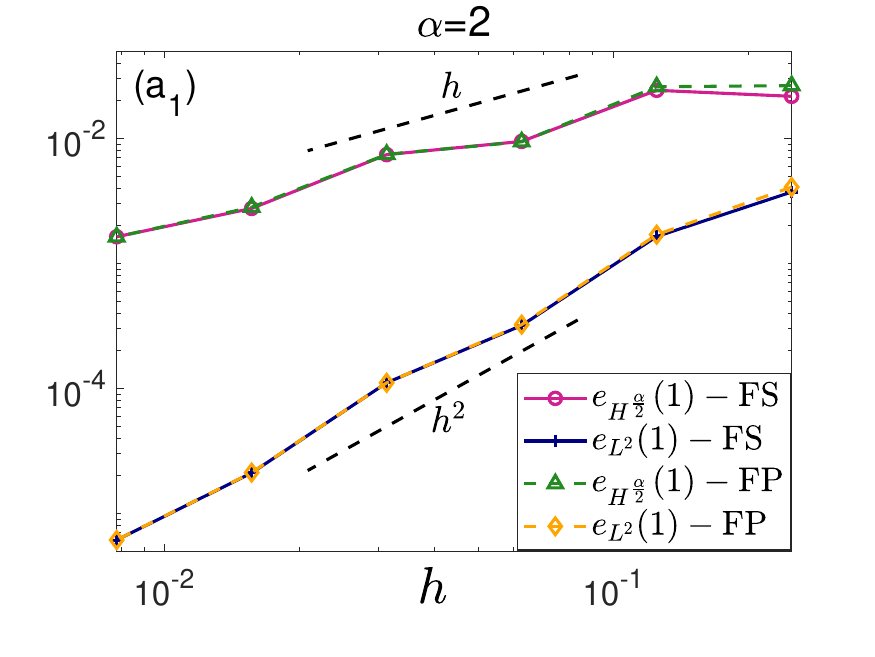}
\includegraphics[scale=.42]{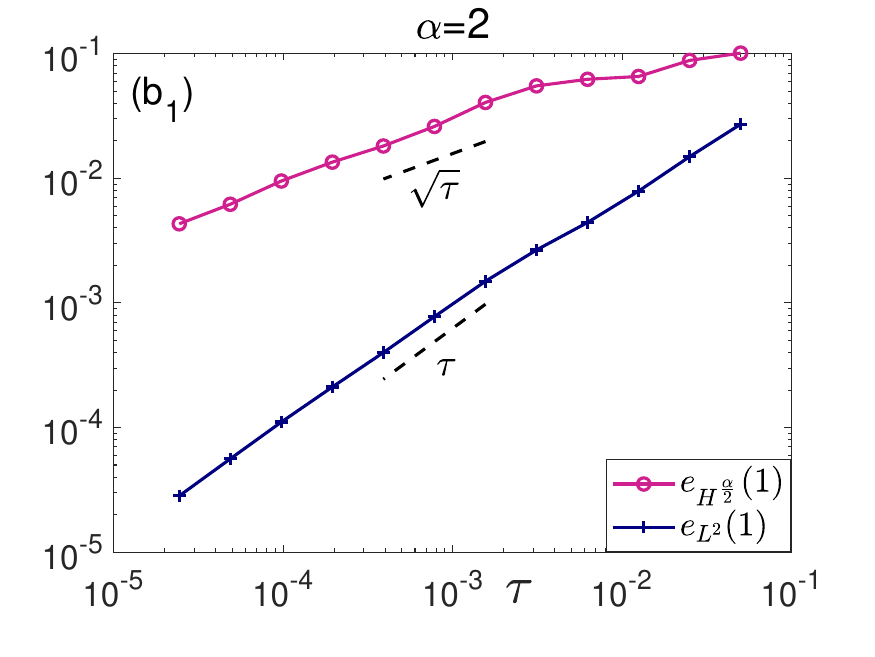}
\includegraphics[scale=.42]{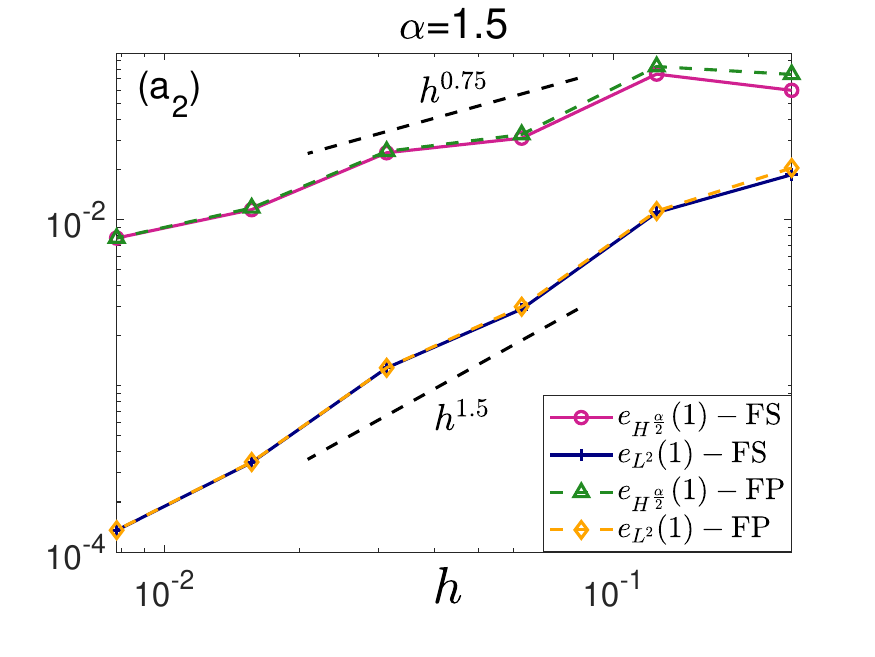}
\includegraphics[scale=.42]{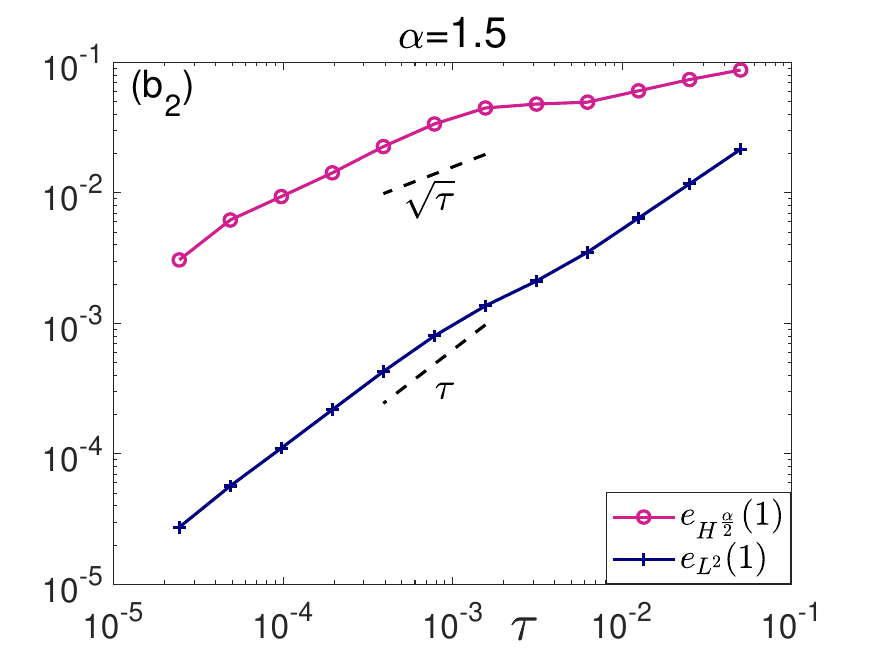}
\includegraphics[scale=.42]{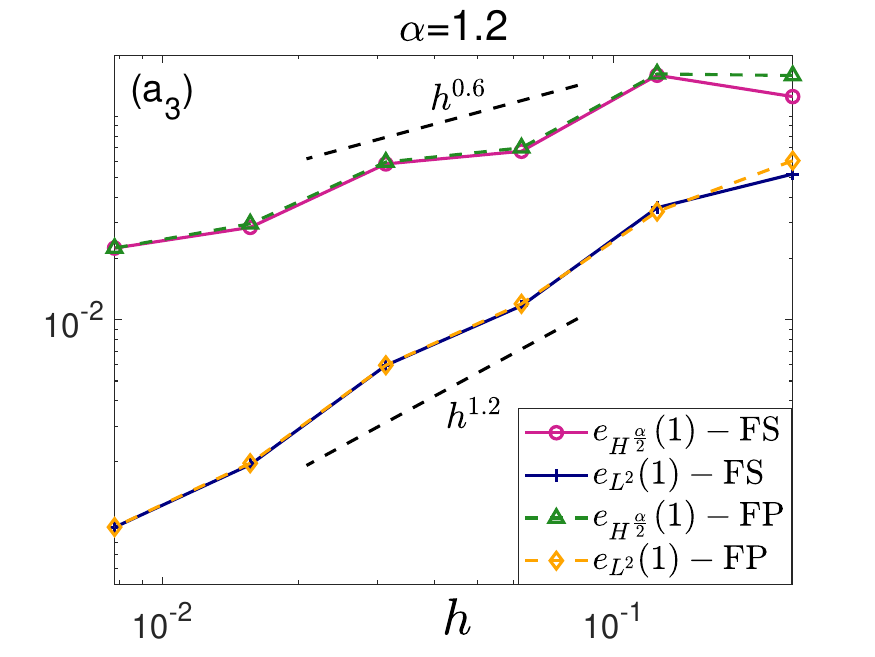}
\includegraphics[scale=.42]{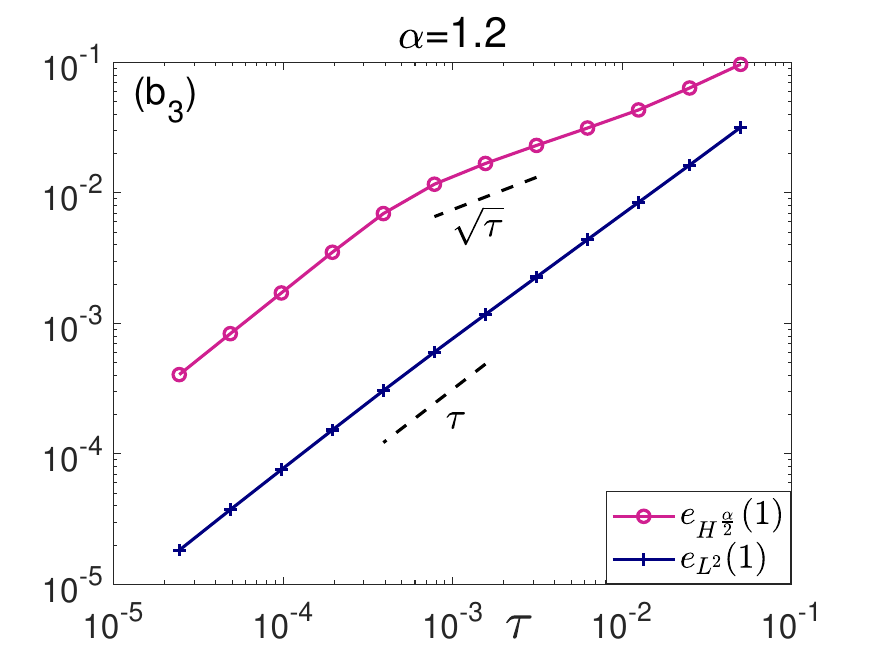}
\caption{Convergence tests of the EWI for \eqref{e5.1n} with low regularity potential $V_2$ and smooth initial condition $\psi_0$ : $(a_{1})-(a_{3})$ spatial errors of the FS and FP discretizations, and $(b_{1})-(b_{3})$ temporal errors in $L^2$-norm and $H^{\frac{\alpha}{2}}$-norm.}
\label{fig2}
\end{figure}
We test the convergence orders for the SFNLSE \eqref{e5.1n} with low regularity potentials $V=V_1$ and $V=V_2$, respectively. The ``exact'' solutions are computed by the EWI-FS method with $\tau=\tau_{\mathrm{e}}:=10^{-6}$ and $h=h_{\mathrm{e}}:=2^{-9}$. When test the spatial errors, we fix the time step size $\tau=\tau_{\mathrm{e}}$, and when test the temporal errors, we fix the mesh size $h=h_{\mathrm{e}}$.

We start with the spatial error and compare the performance of the FS method and the standard FP method.

Figure \ref{fig1} $(a_{1})$-$(a_{3})$ show the spatial error in $L^2$- and $H^{\frac{\alpha}{2}}$-norm of the EWI-FS method (solid lines) and the EWI-FP method (dotted lines) with $V=V_1$ given in \eqref{e5.2n} and smooth initial condition $\psi_0$ given in \eqref{e5.1n}. We can observe that the EWI-FS is $O(h^{2.5})$, $O(h^{2})$, $O(h^{1.7})$ order convergent in $L^2$-norm and $O(h^{1.5})$, $O(h^{1.25})$, $O(h^{1.1})$ order convergent in $H^{\frac{\alpha}{2}}$-norm in space when $\alpha=2, 1.5, 1.2$, repectively. The convergence orders in $L^2$-norm is almost $\frac{\alpha}{2}$ order higher than that of the $H^{\frac{\alpha}{2}}$-norm of the EWI-FS method, this is reasonable. Such convergence orders are not surprising by recalling Theorem \ref{th4.1} since the ``exact'' solution in this case has regularity roughly $H^{\alpha+0.5}$.
However, the spatial convergence order of the EWI-FP method is only first order in both $L^2$- and $H^{\frac{\alpha}{2}}$-norm, and the value of the $H^{\frac{\alpha}{2}}$-error is much larger. This implies that when discretizing low regularity potential, the FS method is much better than the standard FP method. Figure \ref{fig1} $(b_{1})$-$(b_{3})$ plot the temporal convergence of the EWI in $L^2$- and $H^{\frac{\alpha}{2}}$-norm with the Type I low regularity potential. We can observe that the EWI is first order convergent in $L^2$-norm and $H^{\frac{\alpha}{2}}$-norm in time, which is better than the theoretical analysis result. The results in Figure \ref{fig1} validate our optimal $L^2$-norm error bound for the SFNLSE with $L^{\infty}$-potential and demonstrate that it is sharp.

Figure \ref{fig2} $(a_{1})$-$(a_{3})$ show the spatial error in $L^2$- and $H^{\frac{\alpha}{2}}$-norm of the EWI-FS method (solid lines) and the EWI-FP method (dotted lines) with $V=V_2$ given in \eqref{e5.3n} and smooth initial condition $\psi_0$ given in \eqref{e5.1n}. We can observe that the EWI-FS and EWI-FP methods are $O(h^{\alpha})$ order convergent in $L^2$-norm and $O(h^{\frac{\alpha}{2}})$ order convergent in $H^{\frac{\alpha}{2}}$- norm in space. Such convergence orders are also not surprising since the ``exact'' solution in this case has regularity roughly $H^{\alpha}$. However, the spatial error value of the EWI-FP method is slightly larger when $h$ is bigger. But the overall trend shows that the effects of the two methods are similar. Figure \ref{fig2} $(b_{1})$-$(b_{3})$ plot the temporal convergence of the EWI-FS method in $L^2$- and $H^{\frac{\alpha}{2}}$-norm with the initial datum \eqref{e5.1n}. We can observe that the method is $O(\sqrt{\tau})$ order convergent in $H^{\frac{\alpha}{2}}$-norm and $O({\tau})$ order convergent in $L^2$-norm. The results in Figure \ref{fig2} also validate our optimal $H^{\frac{\alpha}{2}}$-norm error bound for the SFNLSE with low regularity potential and demonstrate that it is sharp.

Next, we present some phenomena to demonstrate the difference between the SFNLSE and the classical NLSE with low regularity potential.
\begin{figure}[ht!]
\centering
\includegraphics[scale=.42]{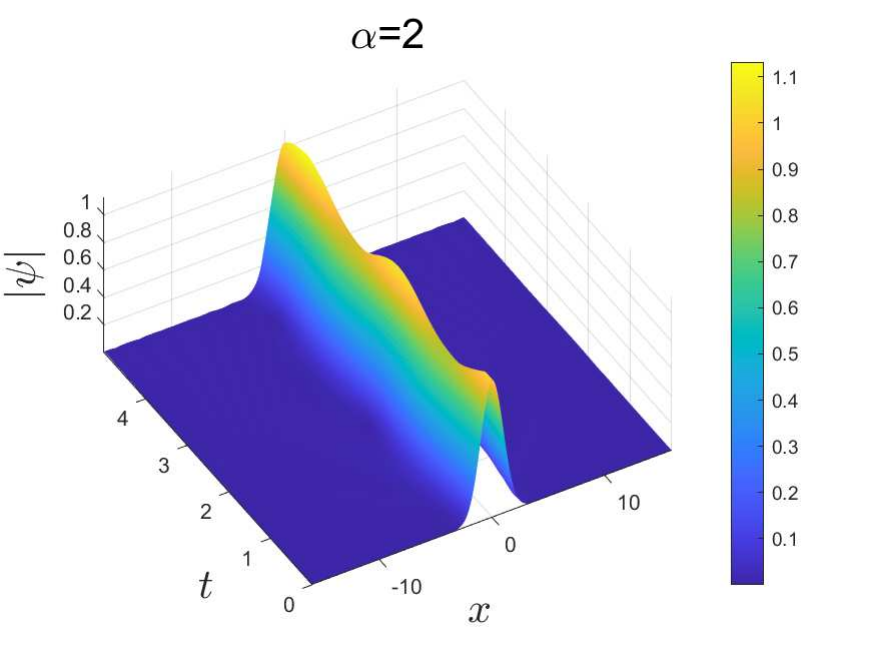}
\includegraphics[scale=.42]{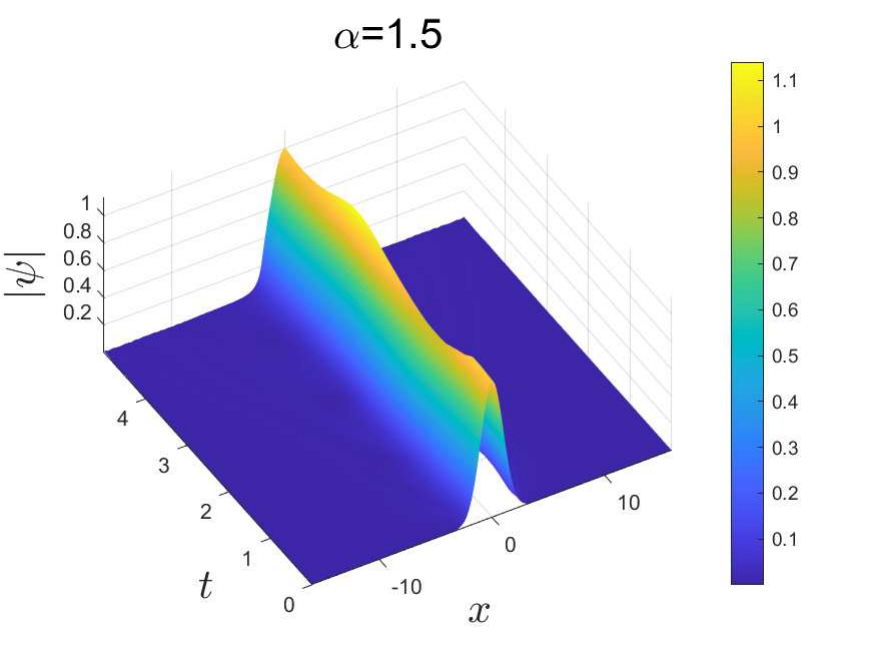}
\includegraphics[scale=.42]{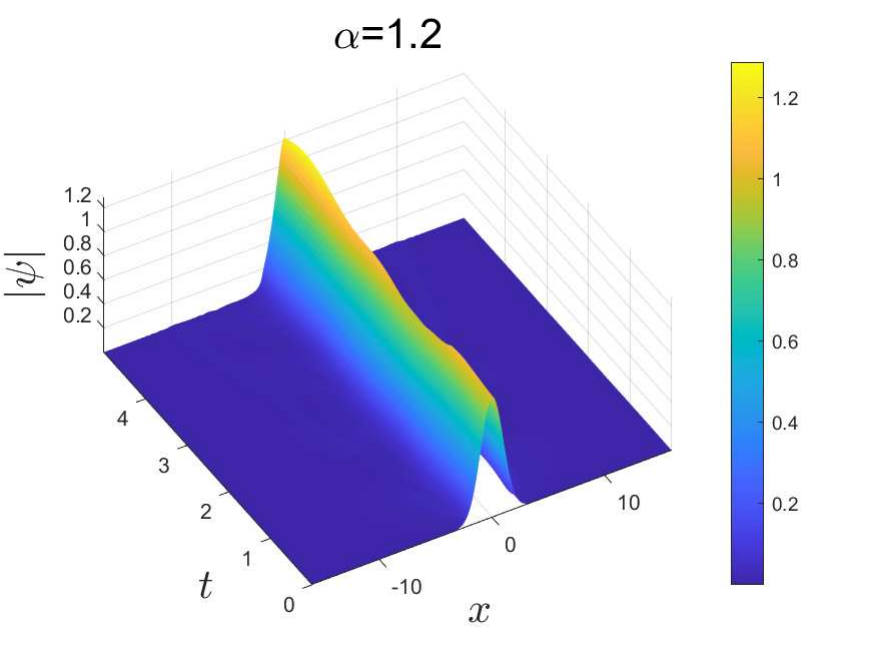}
\caption{Evolutions of $|\psi(x,t)|$ for the SFNLSE \eqref{e5.1n} with low regularity potential $V_1$ and different fractional powers $\alpha=2, 1.5, 1.2$.}
\label{fig3}
\end{figure}

\begin{figure}[ht!]
\centering
\includegraphics[scale=.42]{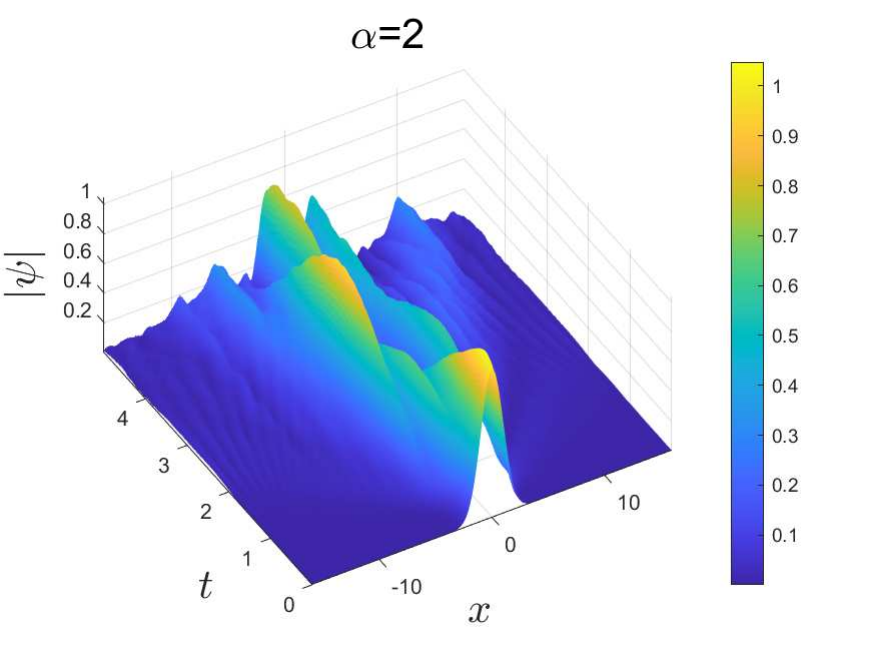}
\includegraphics[scale=.42]{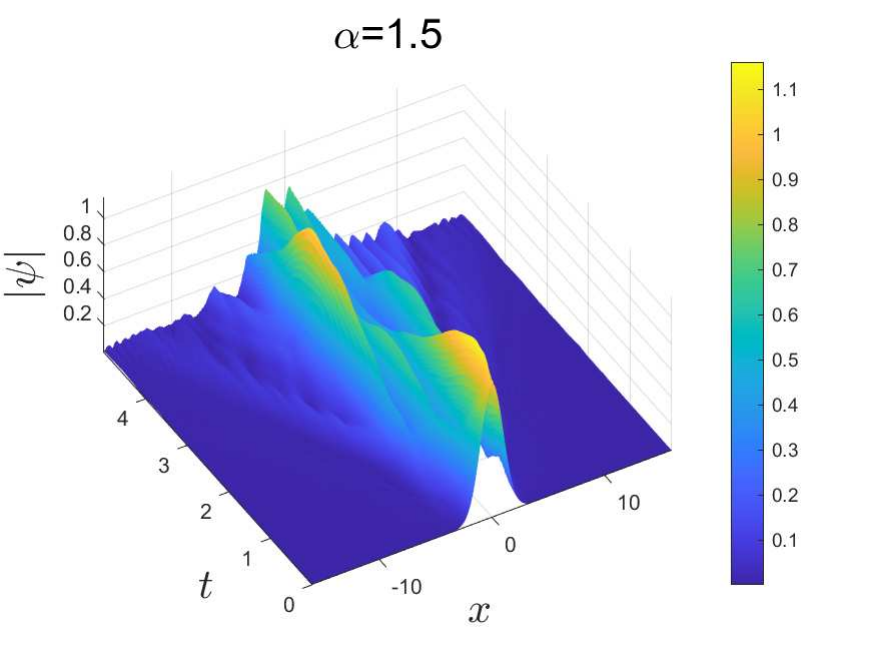}
\includegraphics[scale=.42]{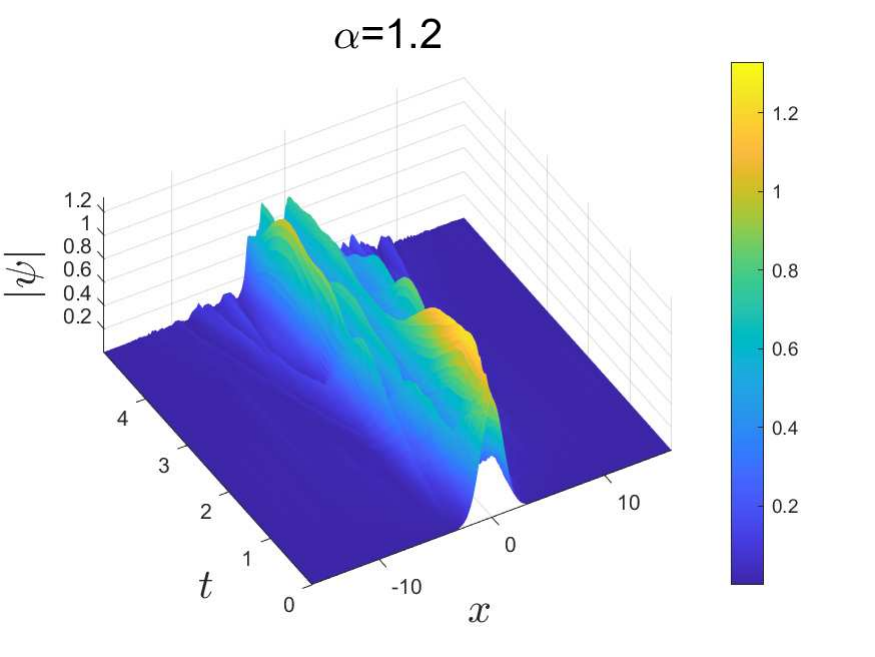}
\caption{Evolutions of $|\psi(x,t)|$ for the SFNLSE \eqref{e5.1n} with low regularity potential $V_2$ and different fractional powers $\alpha=2, 1.5, 1.2$.}
\label{fig4}
\end{figure}

Figures \ref{fig3}-\ref{fig4} show evolutions of $|\psi(x,t)|$ for the SFNLSE with low regularity potentials $V_1$ , $V_2$ and different fractional powers $\alpha=2, 1.5, 1.2$, and we find drastically different evolving patterns between SFNLSE and the classical NLSE. These two figures indicate that when $\alpha$ decreases, the peak height of the numerical solution $|\psi|$ gradually increases. Smaller values of $\alpha$ lead to higher and more concentrated local peaks, while larger values of $\alpha$ result in a more widespread spatial distribution of the solution. With smaller $\alpha$ values, the solution tends to form smooth, localized waves. In both figures, regardless of the different potentials or the value of $\alpha$, the solution evolution patterns are generally similar. They all exhibit a transition from high-frequency, widely distributed oscillations to low-frequency, concentrated distributions. This evolution pattern reflects a consistent influence of $\alpha$ on the overall trend of the solution. And it demonstrates that the long-range interactions and heavy-tailed influence of the fractional diffusion process.

\subsection{The SFNLSE with locally Lipschitz nonlinearity}\label{sub5.2new}
In this subsection, we only consider the SFNLSE with the power-type nonlinearity and without potential:
\begin{equation}\label{e5.1}
i \partial_t \psi(x, t)=(-\Delta)^{\frac{\alpha}{2}} \psi(x, t)-|\psi(x, t)|^{2 \sigma} \psi(x, t), \quad x \in \Omega, \quad t>0,
\end{equation}
where $0<\sigma<1$.  Two types of initial data are considered:\\
(i) Type $\textrm{\uppercase\expandafter{\romannumeral1}}$. The $H^{\alpha}$ initial datum
\begin{equation}\label{e5.2}
\psi_0(x)=x|x|^{\widetilde{\gamma}+0.01}e^{-\frac{x^2}{2}}, \quad x \in \Omega,
\end{equation}
where $\widetilde{\gamma}=\alpha-\frac{3}{2}$.\\
(ii) Type II. The smooth initial datum
\begin{equation}\label{e5.3}
\psi_0(x)=x e^{-\frac{x^2}{2}}, \quad x \in \Omega .
\end{equation}

The two initial data are specifically selected to highlight the impact of the low regularity of $f$ near the origin. The distinction between these two initial data lies in their regularity.

We will test the convergence order in both time and space for Type $\textrm{\uppercase\expandafter{\romannumeral1}}$ and $\textrm{\uppercase\expandafter{\romannumeral2}}$ initial data. For each initial datum, we choose $\sigma=0.1,0.2,0.3,0.4$ or $\sigma=0.1,0.2,0.5,0.6$. The ``exact'' solutions are computed by the Strang splitting FP method with $\tau=\tau_{\mathrm{e}}:=10^{-6}$ and $h=h_{\mathrm{e}}:=2^{-9}$. When test the spatial errors, we fix the time step size $\tau=\tau_{\mathrm{e}}$, and when test the temporal errors, we fix the mesh size $h=h_{\mathrm{e}}$.

\begin{figure}[ht!]
\centering
\includegraphics[scale=.42]{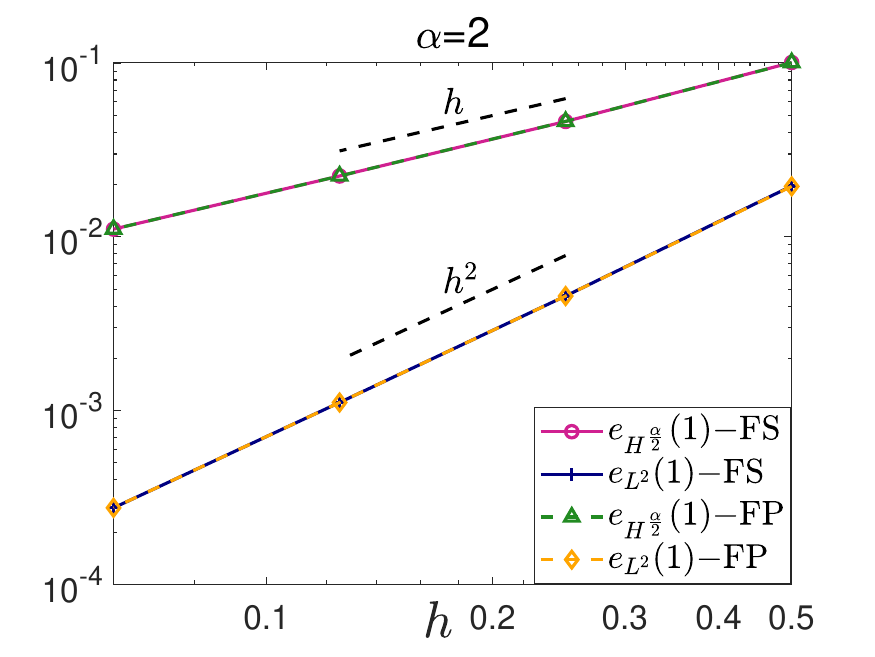}
\includegraphics[scale=.42]{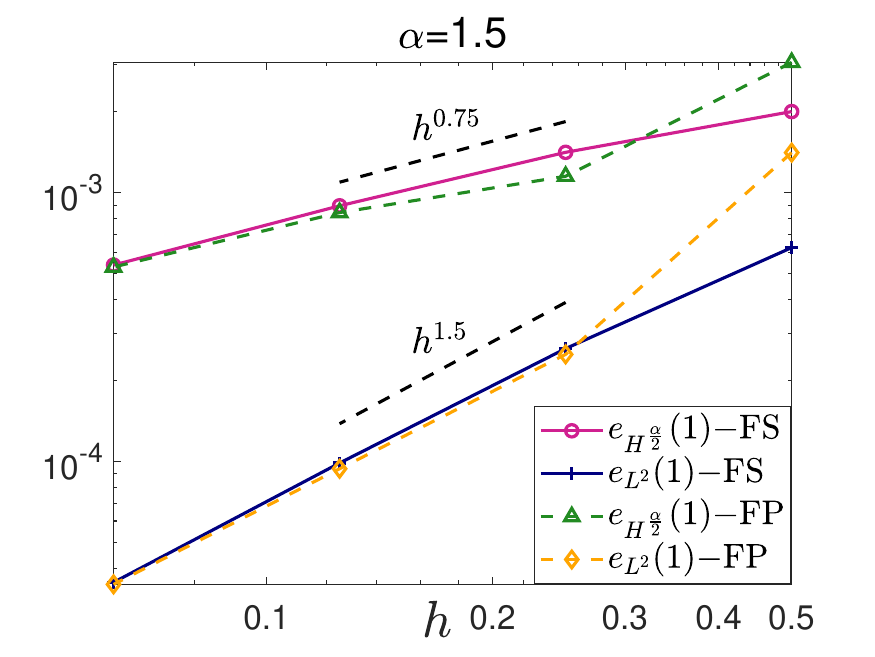}
\includegraphics[scale=.42]{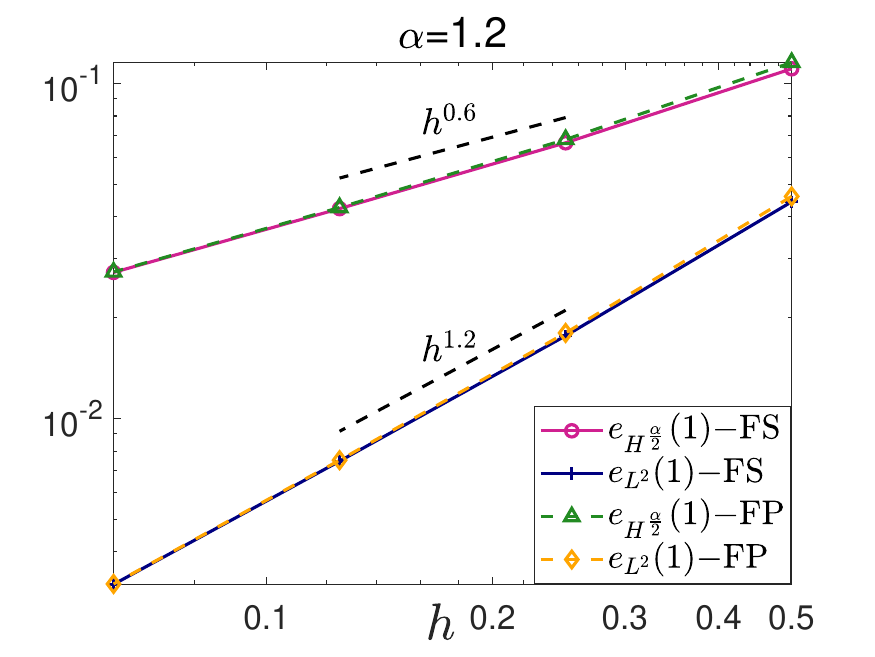}
\caption{Comparison of the FS and FP discretization of the nonlinear term in \eqref{e5.1} with $\sigma=0.1$ and Type I initial datum \eqref{e5.2}.}
\label{fig11}
\end{figure}

We start with the Type I $H^{\alpha}$ initial datum \eqref{e5.2}. Figure \ref{fig11} exhibits the spatial error in $L^2$- and $H^{\frac{\alpha}{2}}$-norm of the EWI-FS (solid lines) and the EWI-FP (dotted lines) method for $\sigma=0.1$. We can observe that the EWI-FS method is $O(h^\alpha)$ order convergent in $L^2$-norm and $O(h^{\frac{\alpha}{2}})$ order convergent in $H^{\frac{\alpha}{2}}$-norm. Moreover, we see that there is almost no difference between the spatial error of the EWI-FS method and the EWI-FP method as the spatial step size
$h$ decreases, which suggests that the FP method seems also suitable to discretize the low regularity nonlinearity.

\begin{figure}[ht!]
\centering
\includegraphics[scale=.42]{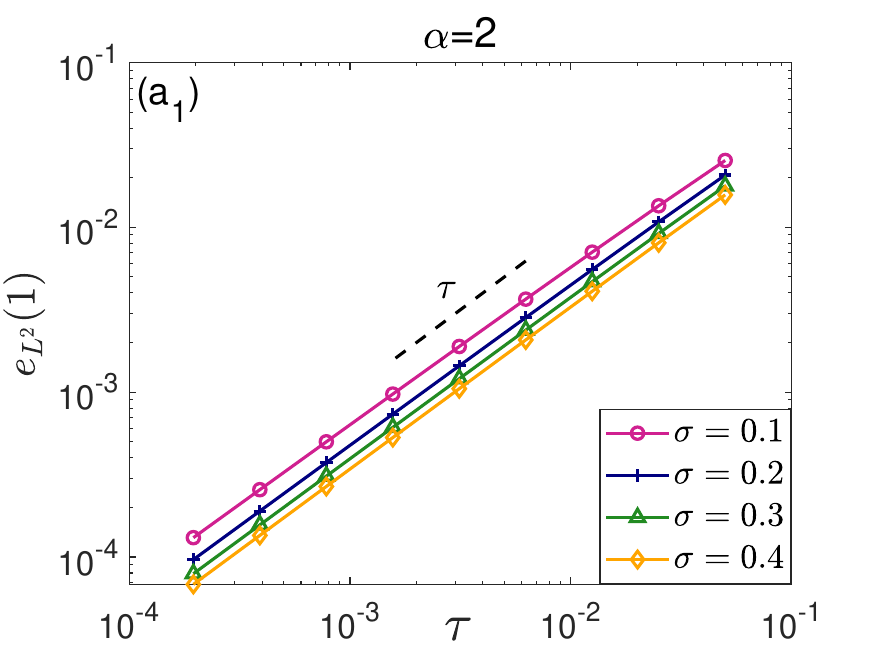}
\includegraphics[scale=.42]{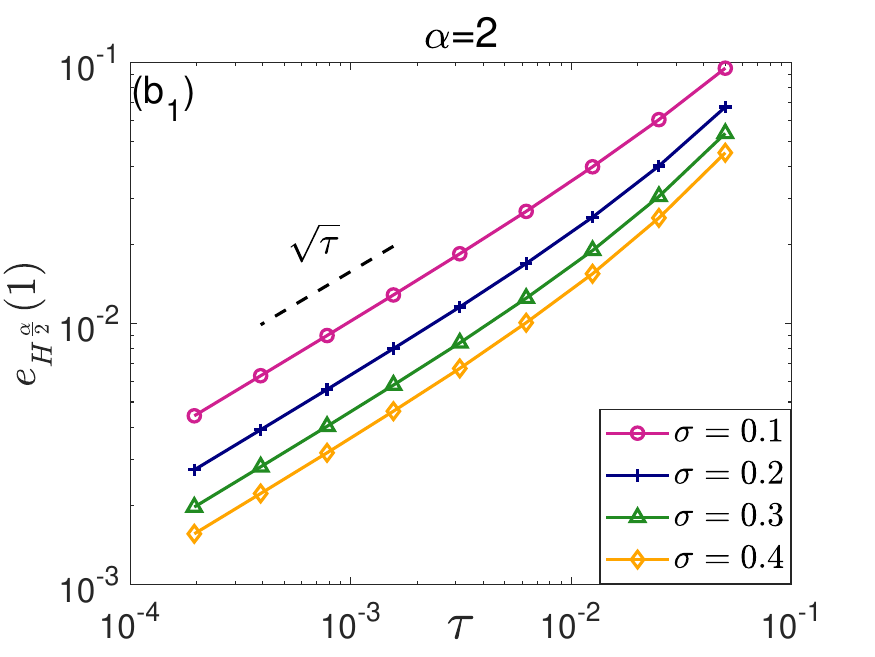}
\includegraphics[scale=.42]{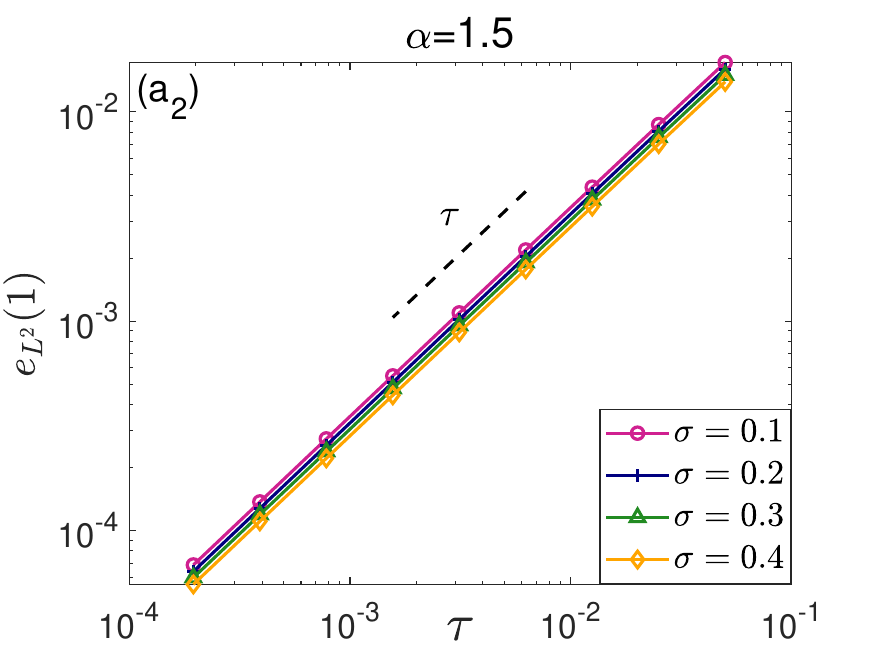}
\includegraphics[scale=.42]{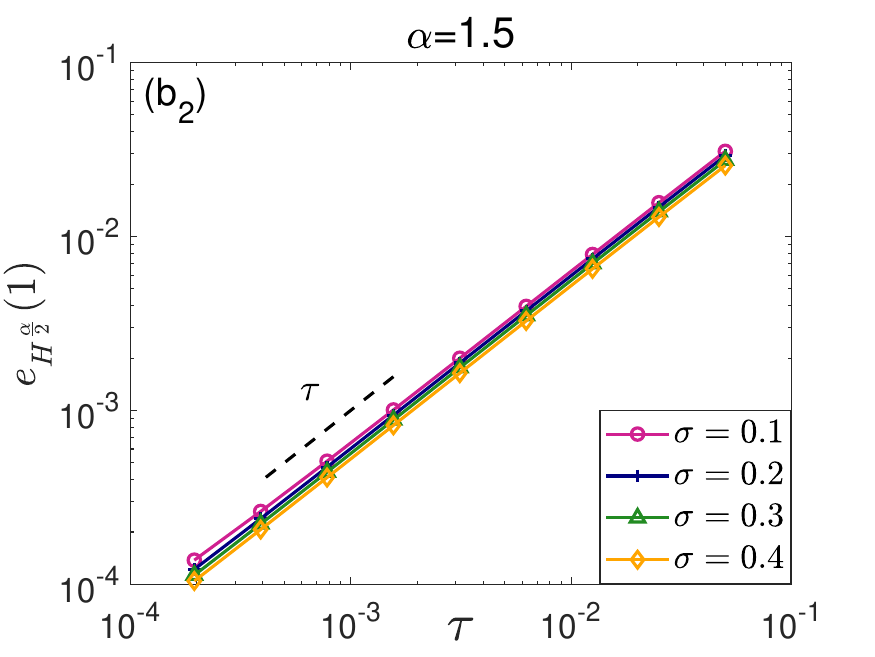}
\includegraphics[scale=.42]{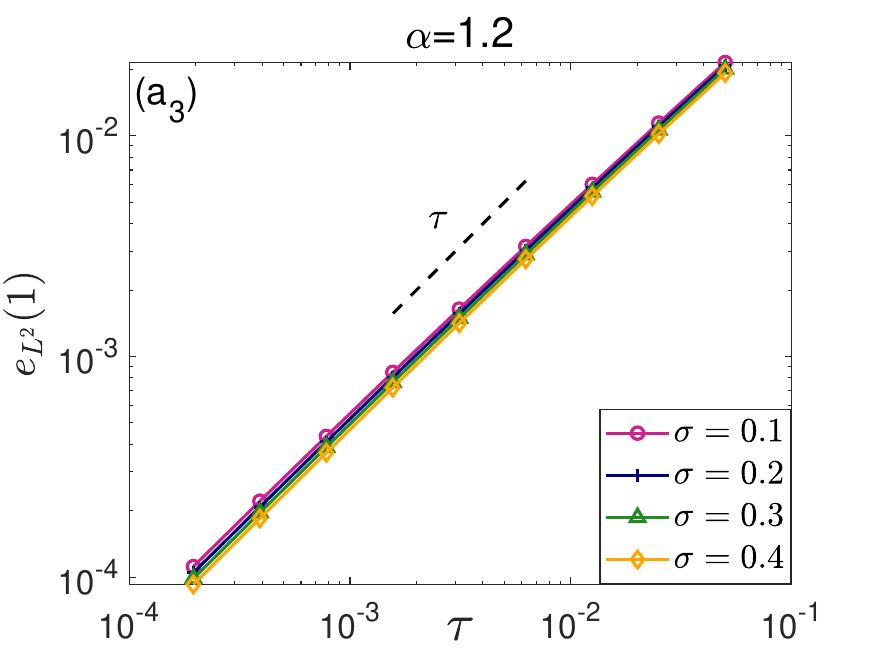}
\includegraphics[scale=.42]{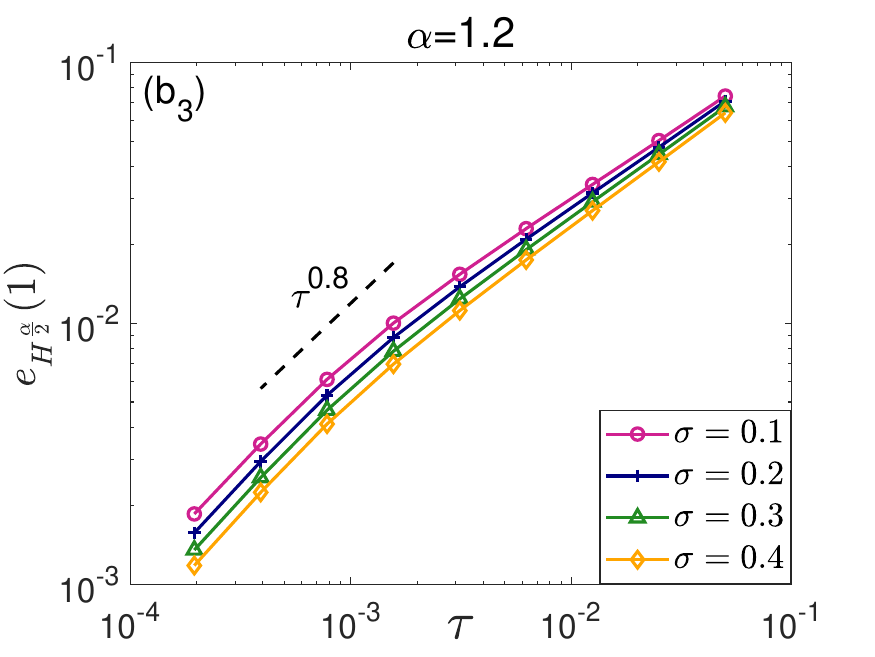}
\caption{Temporal errors of the EWI for the SFNLSE \eqref{e5.1} with Type I initial datum \eqref{e5.2}: $(a_{1})-(a_{3})$ $L^2$-norm errors, and $(b_{1})-(b_{3})$ $H^{\frac{\alpha}{2}}$-norm errors.}
\label{fig22}
\end{figure}

Figure \ref{fig22} plots the temporal error in $L^2$ - and $H^{\frac{\alpha}{2}}$-norm of the EWI for different $0<\sigma<1/2$ with Type I initial datum \eqref{e5.2}. Figure \ref{fig22} $(a_{1})$-$(a_{3})$ show that the temporal convergence order is $O(\tau)$ in $L^2$-norm for $\alpha=2, 1.5, 1.2$. Figure \ref{fig22} $(b_{1})$-$(b_{3})$ show the temporal convergence order is $O(\sqrt{\tau}), O({\tau}), O({\tau^{0.8}})$ in $H^{\frac{\alpha}{2}}$-norm for $\alpha=2, 1.5, 1.2$.
The results in Figures \ref{fig11} and \ref{fig22} confirm our optimal $L^2$-norm error bound for the SFNLSE with locally Lipschitz nonlinearity, and demonstrate that it is sharp.

Then we consider the Type II smooth initial datum \eqref{e5.3}. Figure \ref{fig33} shows the spatial error in $L^2$- and $H^{\frac{\alpha}{2}}$-norm of the EWI-FS (solid lines) and the EWI-FP (dotted lines) method for $\sigma=0.1$ when $\alpha=2, 1.5, 1.2$. We can observe that the convergence orders in $L^2$- and $H^{\frac{\alpha}{2}}$-norm of the EWI-FS and the EWI-FP are almost the same, though the value of the error of the EWI-FS is smaller than the EWI-FP. The results show that when the solution has better regularity, the FS method outperforms the FP method for discretizing the low regularity nonlinearity.

\begin{figure}[ht!]
\centering
\includegraphics[scale=.42]{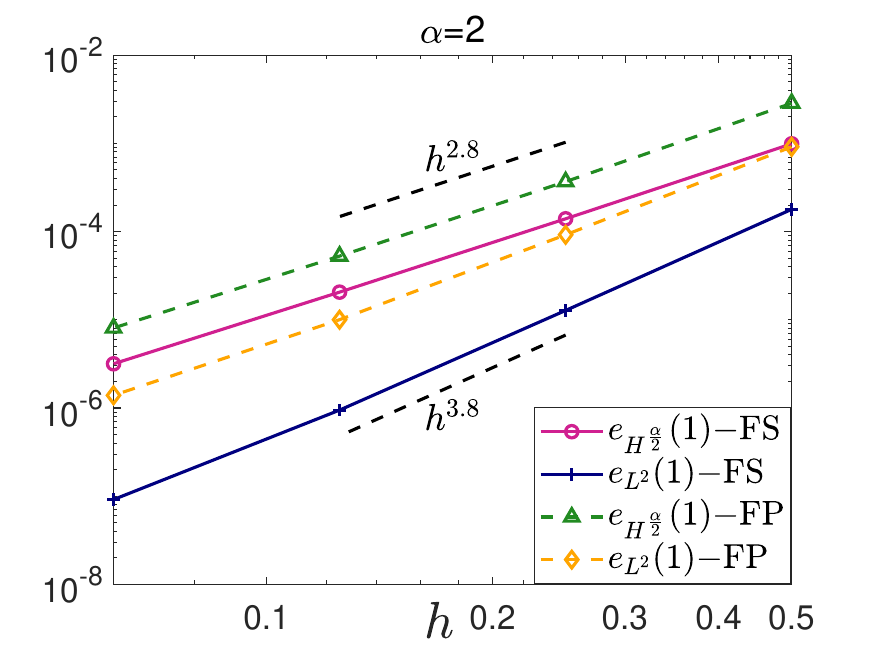}
\includegraphics[scale=.42]{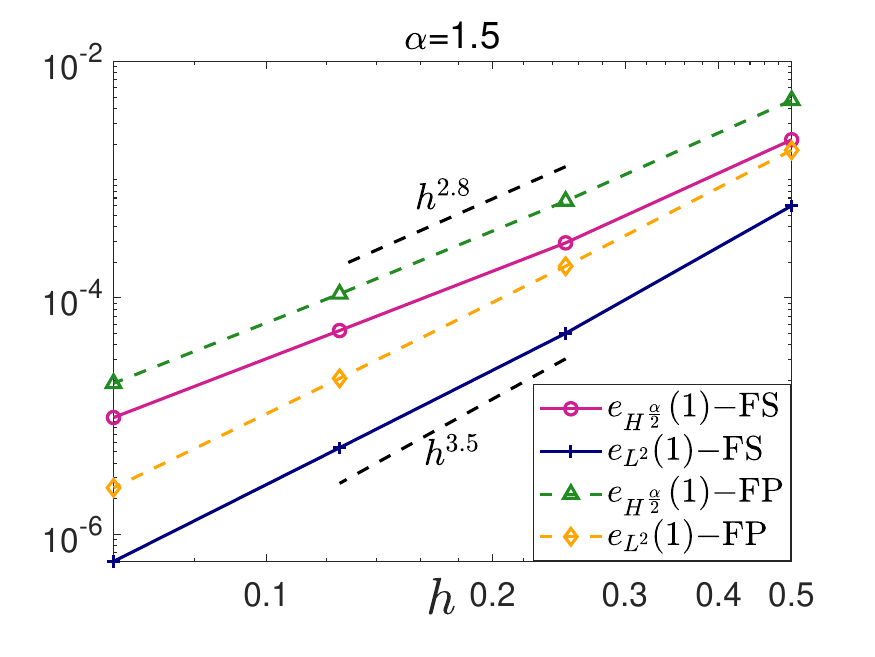}
\includegraphics[scale=.42]{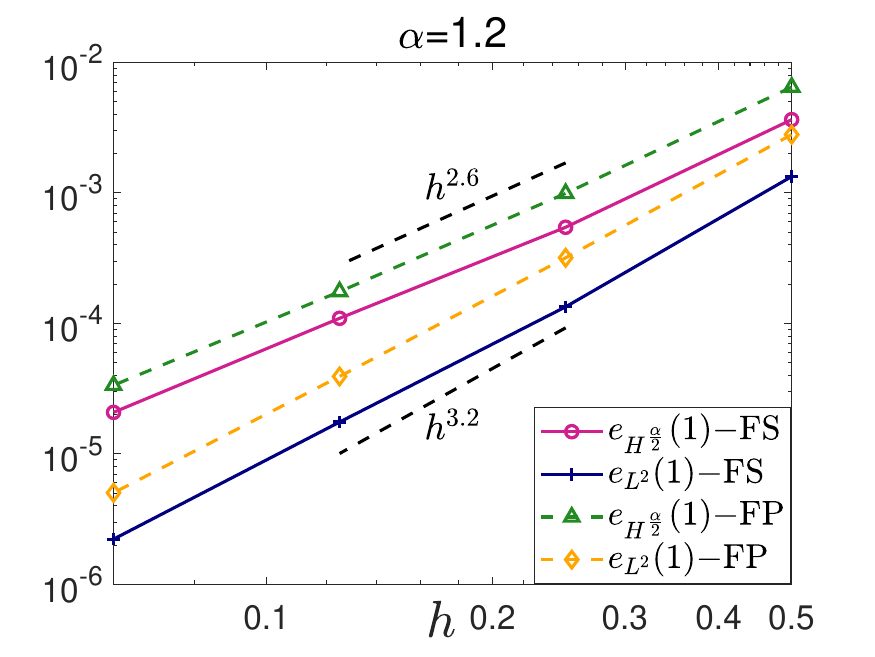}
\caption{Comparison of the FS and FP discretization of the nonlinear term in \eqref{e5.1} with $\sigma=0.1$ and Type II initial datum \eqref{e5.3}.}
\label{fig33}
\end{figure}

\begin{figure}[ht!]
\centering
\includegraphics[scale=.42]{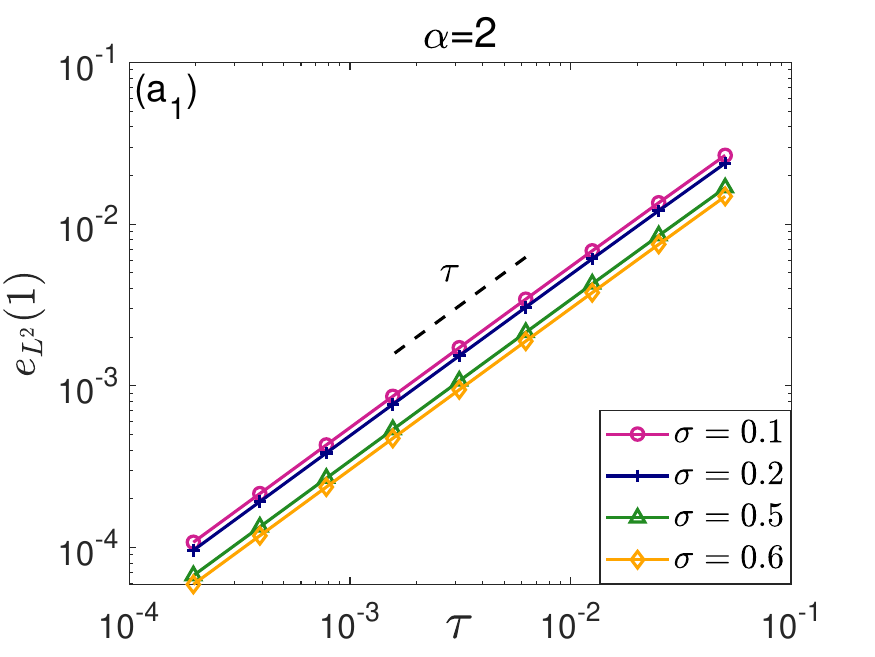}
\includegraphics[scale=.42]{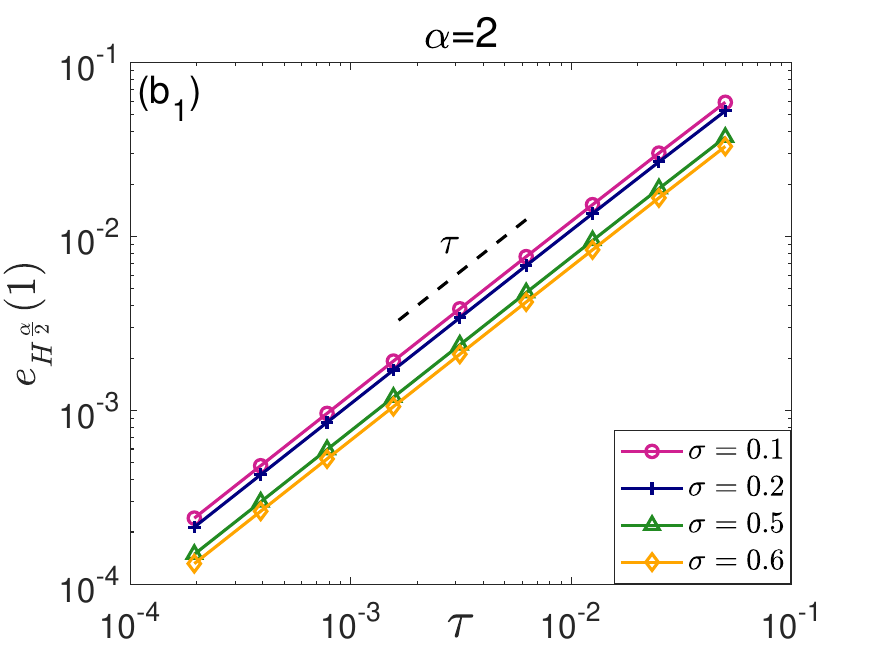}
\includegraphics[scale=.42]{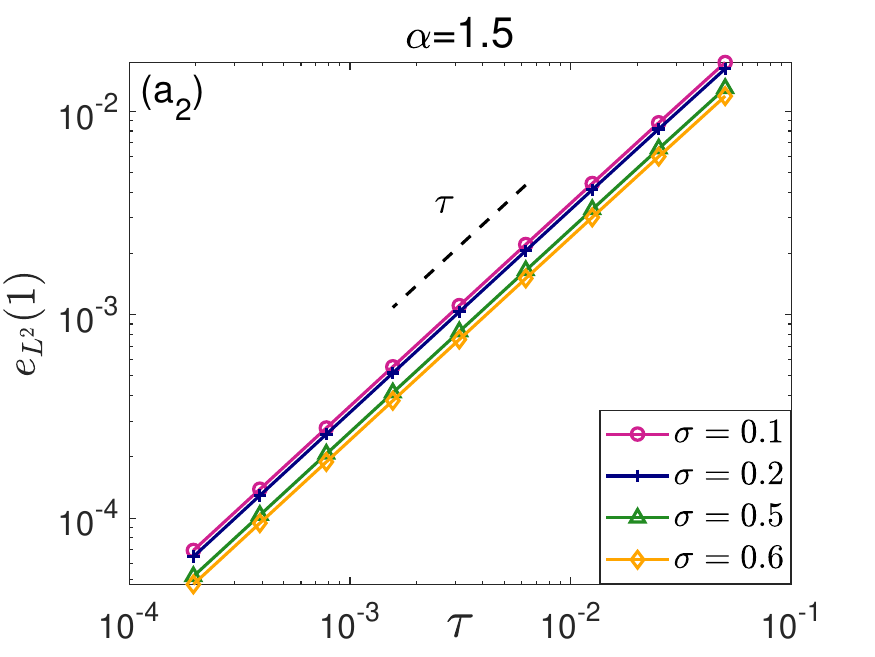}
\includegraphics[scale=.42]{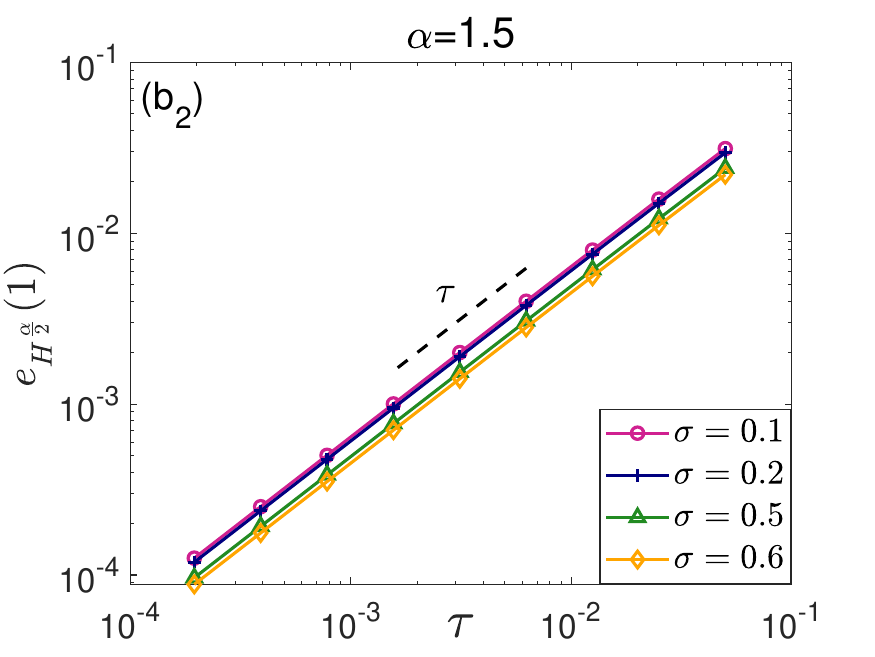}
\includegraphics[scale=.42]{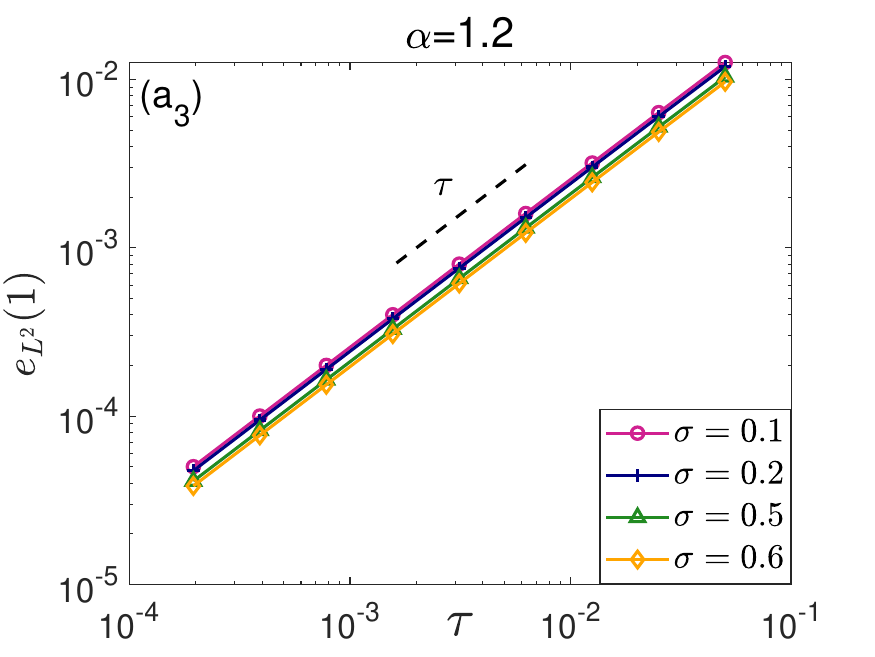}
\includegraphics[scale=.42]{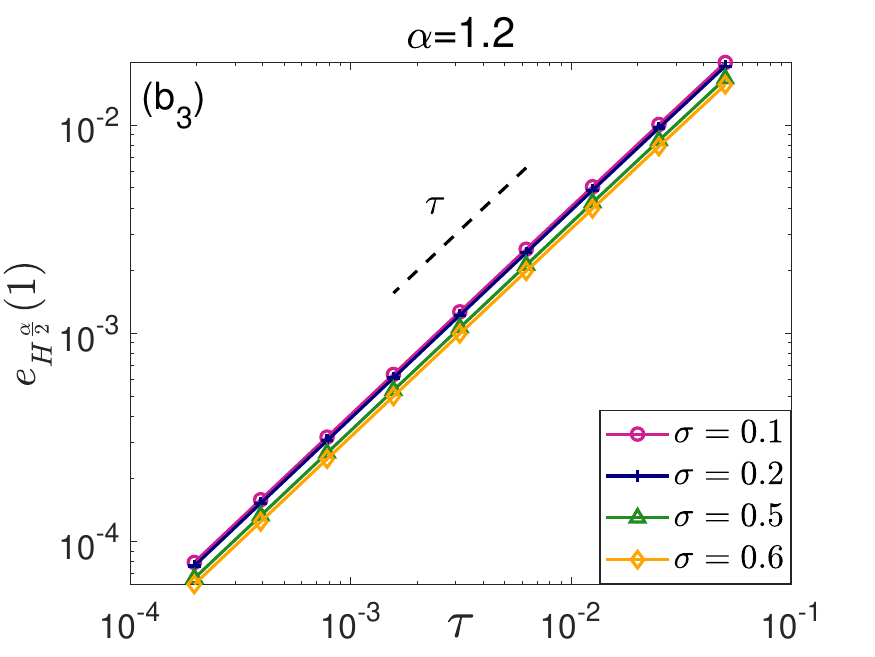}
\caption{Temporal errors of the EWI for the SFNLSE \eqref{e5.1} with Type II initial datum \eqref{e5.3}: $(a_{1})-(a_{3})$ $L^2$-norm errors, and $(b_{1})-(b_{3})$ $H^{\frac{\alpha}{2}}$-norm errors.}
\label{fig44}
\end{figure}

Figure \ref{fig44} displays the temporal error in $L^2$- and $H^{\frac{\alpha}{2}}$-norm of the EWI for different $0<\sigma<1$ with the Type II initial datum. Figure \ref{fig44} $(a_{1})$-$(a_{3})$ and $(b_{1})$-$(b_{3})$ show that the temporal convergence is $O(\tau)$  in both $L^2$ - and $H^{\frac{\alpha}{2}}$-norm for $\sigma=0.1, 0.2, 0.5, 0.6$ when $\alpha=2, 1.5, 1.2$.

The results in Figures \ref{fig33} and \ref{fig44} confirm our optimal $H^{\frac{\alpha}{2}}$-norm error bound for the SFNLSE with low regularity nonlinearity, but also indicates the condition $\sigma > \frac{\alpha}{4}$ in Lemma \ref{assum2} may be relaxed.

Next, we present some phenomena to demonstrate the difference between the SFNLSE and the classical NLSE with low regularity nonlinearity.
\begin{figure}[ht!]
\centering
\includegraphics[scale=.42]{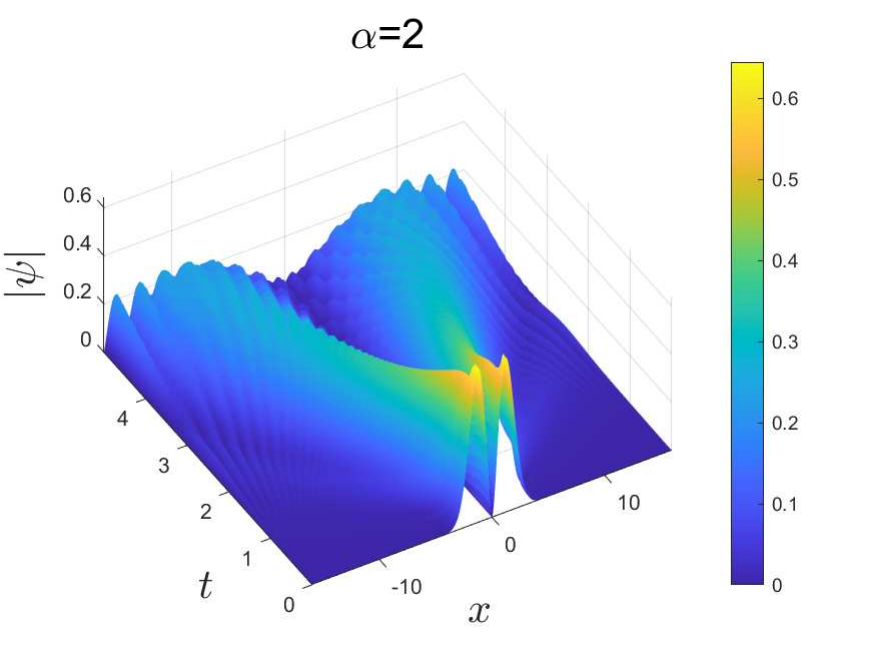}
\includegraphics[scale=.42]{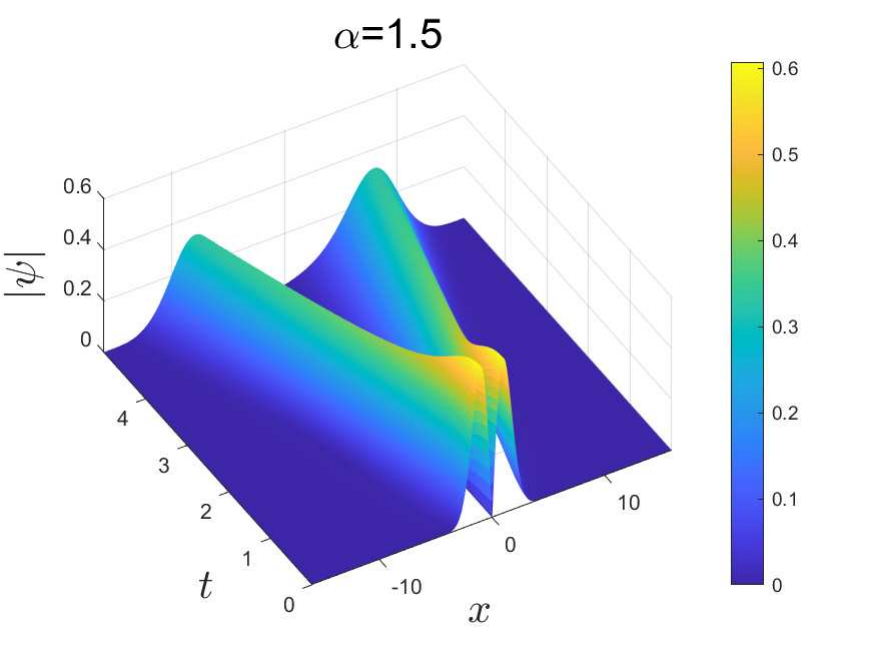}
\includegraphics[scale=.42]{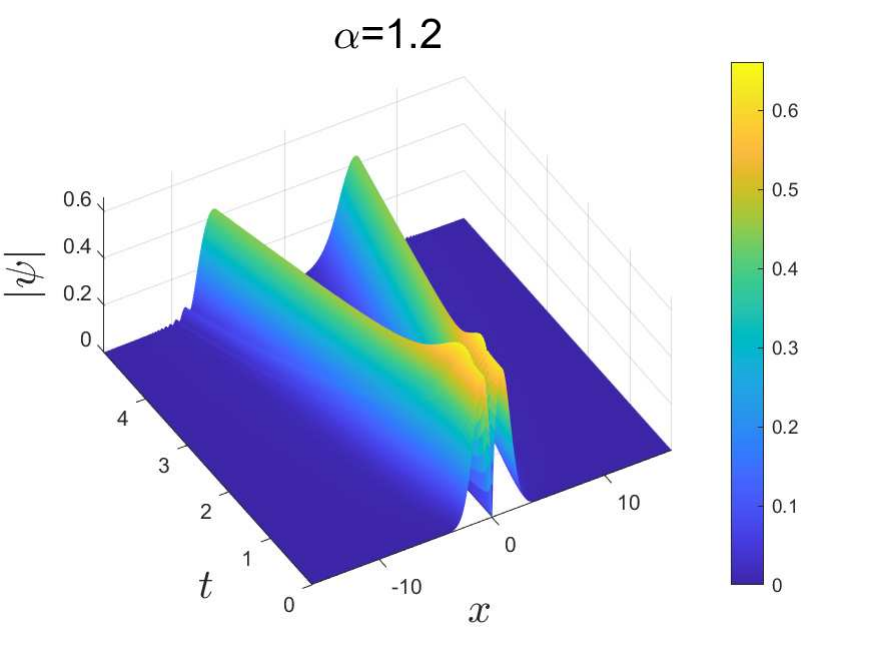}
\caption{Evolutions of $|\psi(x,t)|$ for the SFNLSE \eqref{e5.1} with Type I initial datum \eqref{e5.2} and different fractional powers $\alpha=2, 1.5, 1.2$.}
\label{fig55}
\end{figure}

\begin{figure}[ht!]
\centering
\includegraphics[scale=.42]{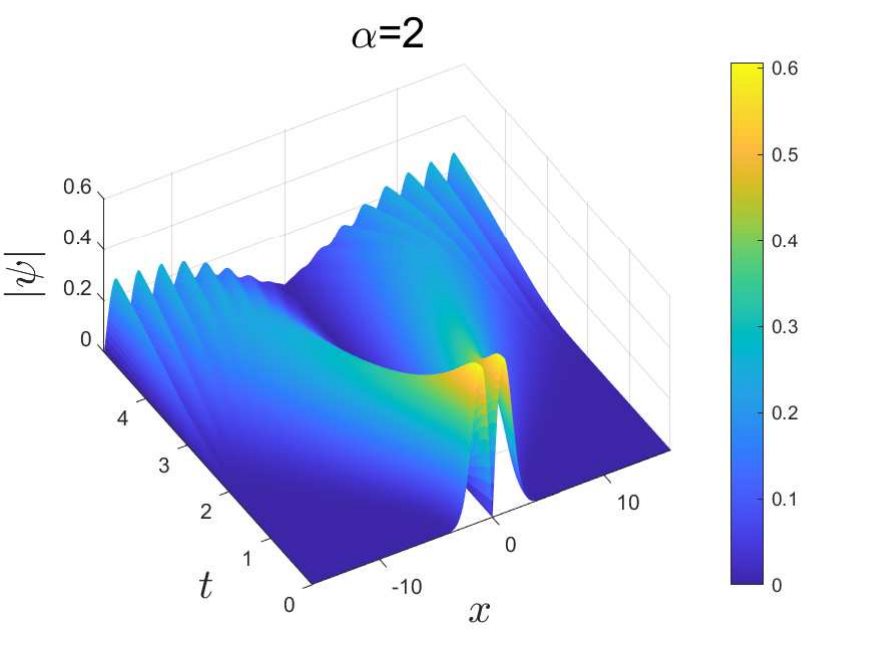}
\includegraphics[scale=.42]{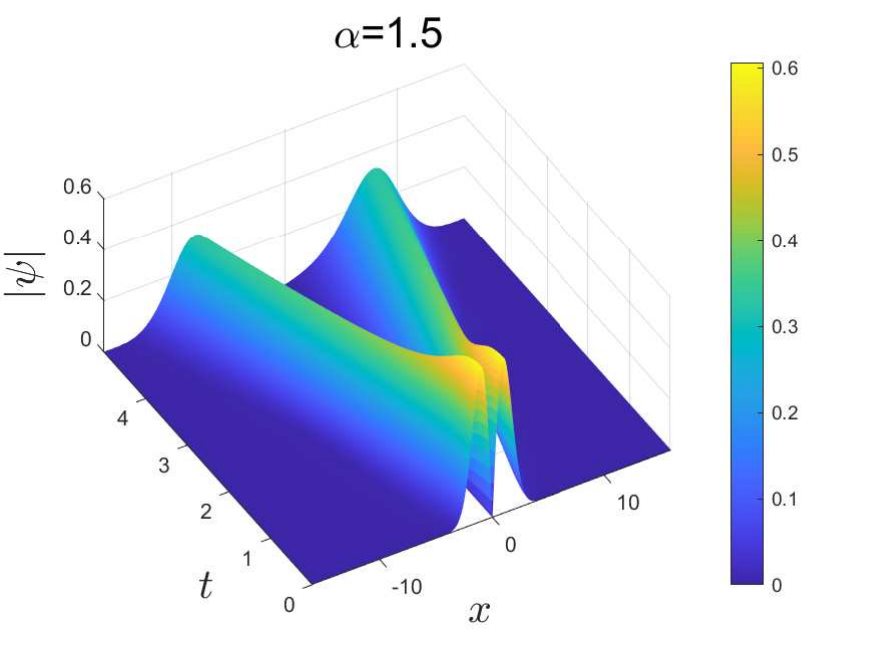}
\includegraphics[scale=.42]{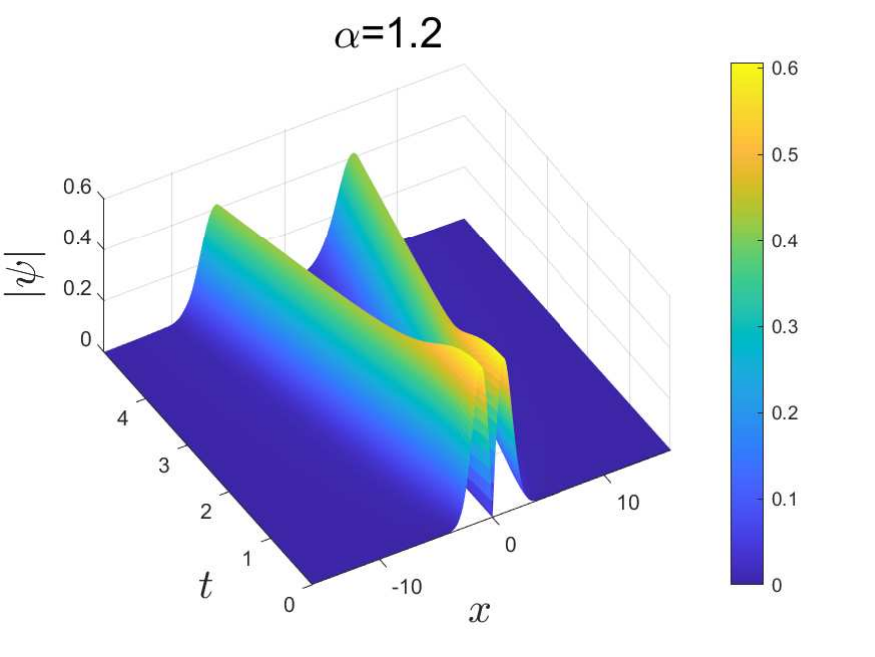}
\caption{Evolutions of $|\psi(x,t)|$ for the SFNLSE \eqref{e5.1} with Type II initial datum \eqref{e5.3} and different fractional powers $\alpha=2, 1.5, 1.2$.}
\label{fig66}
\end{figure}
Figures \ref{fig55}-\ref{fig66} show evolutions of $|\psi(x,t)|$ for the SFNLSE with Type I, II initial datum and different fractional powers $\alpha=2, 1.5, 1.2$. Similar to Figures \ref{fig3}-\ref{fig4} for $\sigma=0.1$, as $\alpha$ decreases, the peak value of the numerical solution $|\psi|$ increases, the oscillations become smoother and more concentrated, and the spatial distribution of the solution exhibits enhanced localization. This indicates that the value of
$\alpha$ has a consistent effect on the local concentration and smoothness of oscillations in the solution across different scenarios.

%%%%%%%%%%%%%%%%%%%%%%%%%%%%%%%%%%%
%\clearpage
\section{Conclusions}
This paper focused on the error estimates on EWI-FS numerical method for the SFNLSE with low regularity potential and/or nonlinearity. We established optimal error bounds on the EWI followed by the FS method for spatial discretization. The first-order Gautschi-type EWI method was used to discretize the temporal direction. And the FS method was imposed to discretize the spatial direction. In practical computations, we use FSwQ method and eFP method for spatial discretization of the EWI to compute the low regularity nonlinearity and potential, respectively. And the computational cost of the method is similar to the standard FP discretization of the EWI. For the semi-discretization in time, we rigorously proved an optimal first-order $L^2$-norm error bound and a uniform $H^\alpha$-norm bound of the semi-discrete numerical solution. For the full discretization, optimal spatial error bounds of $O(h^{m})$ in $L^2$-norm and $O(h^{m-\frac{\alpha}{2}})$ in $H^\frac{\alpha}{2}$-norm were obtained without any CFL condition, and these bounds were optimal with respect to the regularity of the exact solution. Additionally, with slightly stronger regularity assumptions, optimal error bounds at $O(\tau)$ in semi-discretization and $O(\tau+h^{m})$ in full discretization were achieved in $H^\frac{\alpha}{2}$-norm. Extensive numerical examples were conducted to verify the optimal error estimates and confirm their sharpness. The results also revealed significantly distinct evolution patterns between the SFNLSE and the classical NLSE. Additionally, the fractional power $\alpha$ was shown to influence the peak value, oscillatory smoothness, and spatial localization of the solution, illustrating the long-range interactions and heavy-tailed influence of the fractional diffusion process.

\section*{Acknowledgements}
The authors would like to specially thank Professor Weizhu Bao and Dr. Chushan Wang in National University of Singapore for their valuable suggestions and comments. This work has been supported by the National Natural Science Foundation of China (Grants Nos. 12120101001, 12301516), Natural Science Foundation of Shandong Province (Grants Nos. ZR2021ZD03, ZR2024QA159), China Postdoctoral Science Foundation (Grants Nos. GZC20231474, 2024M751788).

\section*{Appendix}\label{appendix}
In the following, we explain the generalization of the current results to 2D and 3D discussed in Remark \ref{rem4.4}.

From \eqref{eq3.8}, using the fractional Leibniz rule (Proposition 1) in \citep{FLeibniz2024}, and the Sobolev embedding (Proposition 1.1) in \citep{Sobolev2013}, as a counterpart of \eqref{eq:gamma-derivative}, we have

\begin{align}
	\|\gamma'(\theta)\|_{H^{\frac{\alpha}{2}}} \leq & \|J^{\frac{\alpha}{2}} f_1(z^\theta)\|_{L^{\frac{2 d}{\alpha}}}\|v-w\|_{L^{\frac{2d}{d-\alpha}}}+\|f_1(z^\theta)\|_{L^{\infty}}\|J^{\frac{\alpha}{2}}(v-w)\|_{L^2} \notag \\
	& + \|J^{\frac{\alpha}{2}} f_2(z^\theta)\|_{L^{\frac{2 d}{\alpha}}}\|\overline{(v-w)}\|_{L^{\frac{2d}{d-\alpha}}}+\|f_2(z^\theta)\|_{L^{\infty}}\|J^{\frac{\alpha}{2}}\overline{(v-w)}\|_{L^2} \notag\\
	\lesssim & (\|J^{\frac{\alpha}{2}} f_1(z^\theta)\|_{L^{\frac{2 d}{\alpha}}}+\|f_1(z^\theta)\|_{L^{\infty}})\|v-w\|_{H^{\frac{\alpha}{2}}},
\end{align}
where $J^{s}=(1-\Delta)^{\frac{s}{2}}, s>0$, by the definition of the $H^{s}$-norm \eqref{eq2.1.3}, we know that $\|J^{s}f\|_{L^2}=\|f\|_{H^{s}}$.

To estimate $\|J^{\frac{\alpha}{2}} f_1(z^\theta)\|_{L^{\frac{2 d}{\alpha}}}$, by the fractional Sobolev embedding again, we have
\begin{equation}
	\|J^{\frac{\alpha}{2}} f_1(z^\theta)\|_{L^{\frac{2 d}{\alpha}}} \lesssim\|J^{\frac{\alpha}{2}} f_1(z^\theta)\|_{H \frac{d-\alpha}{2}} \lesssim \|f_1(z^\theta)\|_{H^{\frac{d}{2}}} .
\end{equation}

For 2D, $\|f_1(z^\theta)\|_{H^{\frac{d}{2}}}=\|f_1(z^\theta)\|_{H^1}$ can be bounded when $\sigma\geq 1/2$.

For 3D, we need $\sigma>3/4$ to estimate $\|f_1(z^\theta)\|_{H^{\frac{3}{2}}}$ by following a similar procedure in \eqref{eq3.9}.
$$f_{1}(v)(\textbf{x})=|v(\textbf{x})|^{2\sigma},$$
where $v=z^\theta$, then
$$\nabla f_{1}(v)(\textbf{x})=(\sigma|v|^{2\sigma-2}v)\nabla\overline{v}(\textbf{x})+(\sigma|v|^{2\sigma-2}\overline{v})\nabla{v}(\textbf{x}),$$
and $\|f_1(v)\|_{H^{\frac{3}{2}}}=\|\nabla f_1(v)\|_{H^{\frac{1}{2}}}$, $\sigma|v|^{2\sigma-2}v$ is $(2\sigma-1)-$H\"{o}lder continuous. Since the second part in the above expression, i.e., $(\sigma|v|^{2\sigma-2}\overline{v})\nabla{v}(\textbf{x})$, is the conjugate of the first part, i.e., $(\sigma|v|^{2\sigma-2}v)\nabla\overline{v}(\textbf{x})$, we only need to estimate the $H^{\frac{1}{2}}$-norm of the first part. Then, we obtain
\begin{equation*}
		\begin{aligned}\label{eq3.9-new2}
			&|\nabla f_{1}(v)|_{H^{\frac{1}{2}}}^2 \\
			&\lesssim\int_{\Omega}\int_{\Omega}\frac{|\nabla f_{1}(v(\textbf{x}+\textbf{y}))-\nabla f_{1}(v(\textbf{y}))|^2}{|\textbf{x}|^{3+1}}\mathrm{d} \textbf{x} \mathrm{d} \textbf{y} \\	&\lesssim\int_{\Omega}\int_{\Omega}\frac{\left|[(\sigma|v|^{2\sigma-2}v)(\textbf{x}+\textbf{y})-(\sigma|v|^{2\sigma-2}v)(\textbf{y})]\nabla\overline{v}(\textbf{x}+\textbf{y})
				+(\sigma|v|^{2\sigma-2}{v})(\textbf{y})(\nabla\overline{v}(\textbf{x}+\textbf{y})-\nabla\overline{v}(\textbf{y}))\right|^2}{|\textbf{x}|^{4}}\mathrm{d} \textbf{x} \mathrm{d} \textbf{y} \\			&\lesssim\int_{\Omega}\int_{\Omega}\frac{|v(\textbf{x}+\textbf{y})-v(\textbf{y})|^{4\sigma-2}\cdot|\nabla\overline{v}(\textbf{x}+\textbf{y})|^{2}}{|\textbf{x}|^{4}}\mathrm{d} \textbf{x} \mathrm{d}\textbf{y} \quad\quad\text{($\sigma|v|^{2\sigma-2}v$ is $(2\sigma-1)-$H\"{o}lder continuous.)} \\
			&\quad+\int_{\Omega}\int_{\Omega}\frac{|(\sigma|v|^{2\sigma-2}v)(\textbf{y})|^{2}\cdot|\nabla\overline{v}(\textbf{x}+\textbf{y})-\nabla\overline{v}(\textbf{y})|^2}{|\textbf{x}|^{4}}\mathrm{d} \textbf{x} \mathrm{d} \textbf{y} \\
			&\lesssim\int_{\Omega}\int_{\Omega}\frac{|\int_{0}^{1}\nabla v(\textbf{y}+\theta \textbf{x})\textbf{x}\mathrm{d}\theta|^{1+\varepsilon}\cdot |\nabla\overline{v}(\textbf{x}+\textbf{y})|^{2}}{|\textbf{x}|^{4}}\mathrm{d} \textbf{x} \mathrm{d} \textbf{y}+\|\nabla v\|^{2}_{H^{\frac{1}{2}}}\quad \text{($0<\theta<1, \sigma>\frac{3}{4}$)} \\
&\lesssim \int_{\Omega}\frac{|\textbf{x}|^{1+\varepsilon}}{|\textbf{x}|^4}\mathrm{d} \textbf{x}\int_{\Omega}\int_{0}^{1}|\nabla v(\textbf{y}+\theta\textbf{x})|^{1+\varepsilon}\cdot|\nabla\overline{v}(\textbf{x}+\textbf{y})|^{2}\mathrm{d}\theta\mathrm{d} \textbf{y}+\|v\|^{2}_{H^{\frac{3}{2}}} \\			
			&\lesssim\int_{\Omega}\frac{1}{|\textbf{x}|^{3-\varepsilon}}\mathrm{d} \textbf{x}\|\nabla v\|_{L^{3+\varepsilon}}^{3+\varepsilon}+\|v\|^{2}_{H^{\frac{3}{2}}}\quad \text{(Using H\"{o}lder's inequality)} \\
&\lesssim \|\nabla v\|_{L^{3+\varepsilon}}^{3+\varepsilon}+\|v\|^{2}_{H^{\frac{3}{2}}}\quad \text{($\varepsilon>0$, the integral $\int_{\Omega}\frac{1}{|\textbf{x}|^{3-\varepsilon}}\mathrm{d} \textbf{x}$ converges in 3D)} \\
			&\lesssim \|\nabla v\|_{H^{\alpha-1}}^{3+\varepsilon}+\|v\|^{2}_{H^{\frac{3}{2}}} \lesssim \|v\|_{H^{\alpha}}^2.			
		\end{aligned}
\end{equation*}
In the final step of the above inequality, by the embedding theorem ($H^{s}\hookrightarrow L^{p}, p=\frac{2d}{d-2s}$), we need $\varepsilon \leq \frac{6\alpha-9}{5-2\alpha}$ to ensure $\|\nabla v\|_{L^{3+\varepsilon}}\lesssim \|\nabla v\|_{H^{\alpha-1}}$. Thus $0<\varepsilon\leq \frac{6\alpha-9}{5-2\alpha}$ and $\frac{3}{2}<\alpha<\frac{5}{2}$.

%\clearpage

\end{document}